\documentclass[11pt]{amsart}

\usepackage{verbatim}
\usepackage{marginnote}
\usepackage{amsmath}
\usepackage{amsfonts}
\usepackage{amssymb}
\usepackage{amsthm}
\usepackage{graphicx}
\usepackage{setspace}
\usepackage{gensymb}
\usepackage{enumerate}
\usepackage{mathrsfs}
\usepackage{empheq} 
\usepackage{graphics} %% add this and next lines if pictures should be in esp format
\usepackage{epsfig} %For pictures: screened artwork should be set up with an 85 or 100 line screen\usepackage[colorlinks=true]{hyperref}
\usepackage[colorlinks=true]{hyperref}

\usepackage{color}
\usepackage{fixmath}
\usepackage{pstricks}
\usepackage{multido}

\DeclareMathOperator{\spt}{spt}

\newcommand{\Lam}{\Lambda}

\newcommand{\tht}{\theta}

\newcommand{\al}{\alpha}
\newcommand{\eps}{\epsilon}
\newcommand{\be}{\beta}
\newcommand{\ph}{\varphi}
\newcommand{\De}{\Delta}
\newcommand{\de}{\delta}
\newcommand{\s}{\sigma}

\newcommand{\gam}{\gamma}

\newcommand{\kap}{\kappa}
\newcommand{\lam}{\lambda}

\newcommand{\N}{\mathbb{N}}
\newcommand{\R}{\mathbb{R}}

\newcommand{\Z}{\mathbb{Z}}
\newcommand{\cZ}{\mathcal{Z}}

\newcommand{\del}{\nabla}
\newcommand{\bdy}{\partial}
\newcommand{\ddt}{\frac{d}{dt}}

\newcommand{\til}[1]{\widetilde{#1}}

\newcommand{\lp}{\triangle}

\newcommand{\lpj}{\triangle_j}
\newcommand{\lpk}{\triangle_k}
\newcommand{\lpl}{\triangle_\ell}
\newcommand{\e}{e^{\lam t^{\al/\kap}\Lam^\al}}

\newcommand{\smod}{\setminus}

\newcommand{\F}{\mathcal{F}}

\newcommand{\sub}{\subset}

\newcommand{\ls}{\lesssim}

\newcommand{\goesto}{\rightarrow}

\newcommand{\hbes}{\dot{B}^s_{p,q}}
\newcommand{\Besb}{\dot{B}^{1+2/p-\kap+\be}_{p,q}}
\newcommand{\Bess}{\dot{B}^{\s+\be}_{p,q}}
\newcommand{\Bessig}{\dot{B}^{\s}_{p,q}}

\newcommand{\Sch}{\mathcal{S}}

\newcommand{\Sob}[2]{\lVert#1\rVert_{#2}}

\newcommand{\abs}[1]{\lvert#1\rvert}
\newcommand{\no}[1]{\lVert#1\rVert}

\newcommand{\lb}{\langle}
\newcommand{\rb}{\rangle}

\newcommand{\req}[1]{(\ref{#1})}

\newtheorem{lem}{Lemma}

\newtheorem{thm}{Theorem}
\newtheorem{prop}[lem]{Proposition}
\newtheorem{cor}[thm]{Corollary}

\newtheorem*{thmcomm}{Theorem 2}

\title[On Gevrey regularity for supercritical SQG]
{On Gevrey regularity of the supercritical SQG equation in critical Besov spaces}

\author{A. Biswas$^{1}$}
\address{$^1$Department of Mathematics and Statistics\\
University of Maryland, Baltimore County\\ Baltimore, MD 21250.}

\author{V. Martinez$^{2}$}
\address{$^2$Department of Mathematics\\
Indiana University\\ Bloomington, IN 47405}

\author{P. Silva$^{3}$}
\address{$^2$Department of Mathematics\\
Indiana University\\ Bloomington, IN 47405}

\address{$\dagger$ corresponding author}

\email[A. Biswas]{abiswas@umbc.edu}
\email[V. R. Martinez]{vinmarti@indiana.edu}
\email[P.S. Silva]{pssilve@indiana.edu}

\begin{document}
\begin{abstract}
In this paper we show that the solution of the supercrtical surface quasi-geostrophic (SQG) equation, starting from initial data in a homogeneous critical Besov space belong to a subanalytic Gevrey class.  In particular, we improve upon the result of Dong and Li in \cite{dong:li:crit:besov}, where they showed that the solutions of Chen-Miao-Zhang (cf. \cite{cmz}) are classical solutions.  We extend the approach of Biswas (cf. \cite{biswas:qg}) to critical, $L^p$-based Besov spaces, and adapt the point of view of Lemarie-Rieusset (cf. \cite{lmr}), who treated the operator arising from applying the analytic Gevrey operator to a product of analytic functions as a bilinear multiplier operator.  In order to obtain $L^p$ bounds, we prove that our bilinear multiplier operator is of Marcinkiewicz type, and show that due to additional localizations inherited from working in Besov spaces, this condition implies boundedness.
\end{abstract}

\maketitle
%SECTION: INTRODUCTION
\section{Introduction}
We consider the two-dimensional dissipative surface quasi-geostrophic (SQG) equation given by
	\begin{align}\label{qg}
		\begin{cases}
			\bdy_t\tht+\Lam^\kap\tht-u\cdotp\del\tht=0,\\
			u=(-R_2\tht,R_1\tht),\\
			\tht(x,0)=\tht_0(x),
		\end{cases}
	\end{align}
where $R_j$ is the $j$-th Riesz transform, and $\Lam^\kap:=(-\De)^{\kap/2}$ for $0<\kap\leq2$.

The study of \req{qg} can be divided into three cases: supercritical ($\kap<1$),  critical ($\kap=1$), and subcritical ($\kap>1$), while the case of no diffusion is called the \textit{inviscid} case.  The QG equation has received much attention over the years since it can be viewed as a toy model for the three-dimensional NSE and Euler equations.   It is also of independent interest as it produces turbulent flows different from those arising from Navier-Stokes or Euler.  For instance, the absence of anomalous dissipation in SQG turbulence has recently been established in \cite{const:tarf:vic2}, in contrast with three-dimensional turbulence where this phenomenon has been observed both numerically and experimentally.

The analytical and numerical study of the inviscid SQG equation was initiated by Constantin, Majda, and Tabak in \cite{cmt}, consequently sparking great interest in the study of SQG.  In \cite{cor:sqg}, C\'ordoba positively settled the conjecture from \cite{cmt} that the formation of a simple type of blow-up could not occur.  In general, however, formation of singularities for solutions of inviscid SQG is still open.   Therefore, much focus has been directed towards studying \req{qg} to explore the role of dissipation in preventing blow-up. 

The well-posedness of the subcritical QG equation was established by Resnick in \cite{resnick}, while the long-term behavior of its solutions were studied in (cf. \cite{const:wu:qgwksol, ju:qgattract}).  Breakthrough in the critical case was met relatively recently in the papers of \cite{caff:vass} and \cite{kis:naz:vol}.  Since then, several proofs of the global regularity problem have been found (cf. \cite{kis:naz, const:vic,const:tarf:vic}).  From these techniques, global well-posedness has also been established in other function spaces such as Sobolev and Besov spaces, (cf. \cite{dong:du, dong:li:crit:besov}).  In spite of this, the global regularity problem for the supercritical case is still open.  While it has been established for the  ``slightly" supercritical case by \cite{dab:kis:sil:vic}, where the dissipation is logarithmically enhanced, only conditional or so-called eventual regularity results are known (cf. \cite{const:wu:qgholder,dab}).

This paper focuses on the supercritical case.  In particular, we establish existence of Gevrey regular (see \req{gev:reg}) solutions to SQG (see Theorem \ref{main:thm}) whose initial data lie only in a homogeneous Besov space (see \req{besov} and \req{besov:norm}).  The study of Gevrey regularity or more generally, higher-order regularity of solutions to critical and subcritical SQG has been previously pursued in (\cite{bae:bis:tad1, biswas:qg, biswas:gev, dong, dong:li:crit:besov, dong:li:subcrit}).  The approach taken here is the one from \cite{biswas:qg}, where it is shown that Gevrey regular solutions to SQG exist starting from critical, homogeneous Sobolev spaces exist.  In particular, the approach from \cite{biswas:qg} is inspired by the classical work of Foias-Temam in \cite{ft}, who introduced the technique of Gevrey to establish analyticity in both space and time of solutions to the Navier-Stokes equations in two and three dimensions.  In order to estimate the nonlinear term in our Besov-space-based Gevrey norm (see \req{gev:reg}), we view the nonlinear term as a bilinear multiplier operator and obtain $L^p\times L^q\goesto L^r$ bounds for this operator, where $1/r=1/p+1/q$ with $1<p<\infty$ and $1\leq q\leq\infty$ (Theorem \ref{axis:thm}).  This point of view of was taken by Lemarie-Rieusset in \cite{lmr}, where spatial analyticity of solutions to the NSE starting from $L^p$ initial data was established.  Due to supercritical dissipation, we establish commutator estimates (Theorem \ref{main:comm:thm}) similar to those found in \cite{biswas:qg}, where the commutator considered in \cite{miura} is modified to account for Gevrey regularity.

The notations and conventions used throughout the paper are introduced in Section \ref{sect:notation}, while the statements of our main theorems are located in Section \ref{sect:results}.  We establish our commutator estimate in Section \ref{sect:comm} and Gevrey regularity of solutions to SQG in Section \ref{sect:apriori}.  The proof of our multiplier theorem is relegated to the Appendix (Section \ref{sect:app}).

%SECTION: NOTATION
\section{Notation}\label{sect:notation}

\subsection{Littlewood-Paley decomposition and related inequalities}
Let $\psi_0$ be a radial bump function such that $\psi_0(\xi)=1$ when $[|\xi|\leq1/2]\sub\R^d$, and
	\begin{align}\notag
		0\leq\psi_0\leq1\ \text{and}\ \spt\psi_0=[|\xi|\leq1].
	\end{align}
Define $\phi_0(\xi):=\psi_0(\xi/2)-\psi_0(\xi)$.  Observe that
	\begin{align}\notag
		0\leq\ph_0\leq 1\ \text{and}\ \spt\phi_0=[2^{-1}\leq|\xi|\leq2]
	\end{align}
Now for each $j\in\Z$, define $\psi_j:=(\psi_0)_{2^{-j}}$ and $\ph_j:=(\ph_0)_{2^{-j}}$, where we use the notation
	\begin{align}\label{dilations}
		f_\lam(x):=f(\lam x).
	\end{align}
for any $\lam\geq0$.  Then obviously $\ph_0:=\psi_1-\psi_0$ and $\psi_{j+1}=\psi_j+\ph_j$, so that
	\begin{align}\label{lp:loc}
		\spt\psi_j=[|\xi|\leq2^{j-1}]\ \text{and}\ \spt\ph_j=[2^{j-1}\leq|\xi|\leq2^{j+1}].
	\end{align}
Moreover, we have 
	\begin{align}\notag
		\sum_{j\geq k}\ph_j=1\ \text{for }\ \xi\in\R^d\smod\{\mathbf{0}\}
	\end{align}
One can then show that $f=\lim_{k\goesto\infty}S_kf=:\sum_{k\in\Z}\lpk f$, for any $f\in\Sch'(\R^d)$, where $\Sch'$ is the space of tempered distributions, and 
	\begin{align}
		\lpk f:&=\check{\ph}_k*f,\notag\\
		\til{\lp}_kf:&=\sum_{|k-\ell|\leq 2}\lpl f\notag\\
		S_kf:&=\sum_{\ell\leq k-3}\lpl f\notag
	\end{align}
We call the operators $\lpk$ Littlewood-Paley blocks.  For convenience, we will sometimes use the shorthand $f_k:=\lpk f$.

For functions whose spectral support is compact, one has the Bernstein inequalities, which we will use throughout the paper.  We state it here in terms of Littlewood-Paley blocks.

\begin{lem}[Bernstein inequalities]\label{bern}
Let $1\leq p\leq q\leq\infty$ and $f\in\Sch'(\R^d)$.  Then
	\begin{align}
		2^{js}\Sob{\lpj f}{L^q}\ls&\Sob{\Lam^s \lpj f}{L^q}\ls2^{js+d(1/p-1/q)}\Sob{\lpj f}{L^p},
	\end{align}
for each $j\in\Z$ and $s\in\R$.
\end{lem}

Since we will be working with $L^p$ norms, we will also require the generalized Bernstein inequalities, which was proved in \cite{wu:lowerbounds} and \cite{cmz}.

\begin{lem}[Generalized Bernstein inequalities]\label{gen:bern}
Let $2\leq p\leq\infty$ and $f\in\Sch'(\R^d)$.  Then
	\begin{align}
		2^{j\frac{2sj}p}\Sob{\lpj f}{L^p}\ls&\Sob{\Lam^s \abs{\lpj f}^{p/2}}{L^2}^{\frac{2}p}\ls2^{\frac{2sj}p}\Sob{\lpj f}{L^p},
	\end{align}
for each $j\in\Z$ and $s\in[0,1]$.
\end{lem}

In order to apply these inequalities, we will first need the following positivity lemma, which was initially proved in \cite{cor:cor}, and generalized by Ju in  \cite{ju:qgattract} (see also \cite{const:vic}, \cite{const:tarf:vic}).

\begin{lem}[Positivity lemma]\label{pos:lem}
Let $2\leq p\leq\infty$, $f,\Lam^sf\in L^p(\R^2)$.  Then
	\begin{align}
		\int \Lam^sf\abs{f}^{p-2}f\ dx\geq\frac{2}p\int(\Lam^{\frac{s}2}\abs{f}^{\frac{p}2})^2\ dx.
	\end{align}
\end{lem}

\subsection{Besov spaces}
Let $s\in\R$, $1\leq p,q\leq\infty$.  The \textit{homogeneous Besov space} $\hbes$ is the space defined by
	\begin{align}\label{besov}
		\hbes:=\{f\in\cZ'(\R^d):\Sob{f}{\hbes}<\infty\},
	\end{align}
where $\cZ'(\R^d)$ denotes the dual space of $\cZ(\R^d):=\{f\in\Sch(\R^d):\bdy^\be\hat{f}(0)=0,\forall\be\in\N^d\}$, and the norm is given by
	\begin{align}\label{besov:norm}
		\Sob{f}{\hbes}:=\left(\sum_{j\in\Z}2^{jsq}\Sob{\lpj f}{L^p}^q\right)^{1/q},
	\end{align}
for $1\leq q<\infty$.  One makes the usual modification for $q=\infty$.  For more details, see \cite{bcd} or \cite{runst:sick}.

\subsection{Gevrey operator and related spaces}
Let $0<\al\leq1$ and $\gam>0$.  We denote the \textit{Gevrey operator} by the linear multiplier operator $T_{G_\gam}=\F^{-1}G_\gam\F$ where
	\begin{align}\label{gevrey:op}
		G_\gam(\xi):=\exp(\gam\no{\xi}^\al),
	\end{align}
where $\no{\hspace{0.02in}\cdotp}$ denotes the two-dimensional Euclidean norm.  Note that this notation is not to be confused with $f_\lam$ as defined in \req{dilations}, though the meaning of this notation will be clear from the context.  For convenience, we write the multiplier operator given by $T_{G_\gam}f$ simply as $G_\gam f$ or $\til{f}:=G_\gam f$.

We say that a function $f$ is \textit{Gevrey regular} if
	\begin{align}\label{gev:reg}
		\Sob{G_\gam f}{\hbes}<\infty,
	\end{align}
for some $\gam>0$.  Note that for $p=q=2$, one has the usual notion of Gevrey regularity.

We will establish existence of Gevrey regular solutions to \req{qg} in the following space
	\begin{align}\label{exist:space}
		X_{T}:=\{v\in C((0,T];\Besb(\R^2)):\Sob{v}{X_T}<\infty\},
	\end{align}
for some $0<T\leq\infty$, where $2\leq p<\infty$, $0<\kap\leq1$, and 
	\begin{align}
		\Sob{v}{X_T}:=\sup_{0<t\leq T}t^{\be/\kap}\Sob{G_\gam v(\ \cdotp,t)}{\Besb},
	\end{align}
and any $\be<\kap/2$.

%SECTION: MAIN RESULTS
\section{Main Results}\label{sect:results}
\begin{thm}\label{main:thm}
Let $2\leq p<\infty$, $1\leq q\leq\infty$, and $0<\al<\kap<1$.  Suppose $\tht_0\in\Bessig(\R^2)$, where $\s:=1+2/p-\kap$.  Then there exists $T^*<\infty$ and $\tht\in C([0,T^*);B^\s_{p,q}(\R^2))$ such that $\tht$ satisfies \req{qg} and
	\begin{align}
		\Sob{\tht(\ \cdotp)}{X_T}\ls\Sob{\tht_0}{\Bessig},
	\end{align}
for any $\be<\kap/2$. %, where $\bar{\be}=\bar{\be}(p,\kap)$ defined by
%	\begin{align}\label{beta:def}
%		\bar{\be}:=\begin{cases}
%				\kap/2& \kap\leq1/2+1/p,\\
%				\kap-1/2-1/p+\eps&\kap>1/2+1/p,
%			\end{cases}
%	\end{align}
%for any $\eps<(1-\kap)/2+1/p$.  
Moreover, there exists $C>0$ such that if $\Sob{\tht_0}{\Bessig}\leq C$, then $T^*=\infty$. 
\end{thm}

%\begin{rmk}
%Note that $X_T$ depends on the parameter $\be$ (see \req{exist:space}).  Also, observe that \req{beta:def} implies that $\be<\kap/2$.  If $p=2$, then this is the only restriction on $\be$, which is precisely the one found in \cite{biswas:qg}.
%\end{rmk}

The proof of Theorem \ref{main:thm} will make use of the following commutator estimate for Gevrey regular functions.

\begin{thm}\label{main:comm:thm}
Let $1< p<\infty$ and $1\leq q\leq\infty$.  Let $\gam,\de>0$ such that $\de<1$.  Suppose $s,t\in\R$ satisfy the following
	\begin{enumerate}
		\item[(i)]$2/p<s<1+2/p-\de$,
		\item[(ii)] $t<2/p$,
		\item[(iii)] $s+t>2/p$.
	\end{enumerate}
Then there exists $C_j=C_j(\al,\de,\gam)$ such that 
	\begin{align}\notag
		\Sob{[G_{\gam}\lpj,f]g}{L^p(\R^2)}\ls2^{-(s+t-2/p)j}C_j\Sob{G_\gam{f}}{\dot{B}_{p,q}^{s}(\R^2)}\Sob{G_\gam{g}}{\dot{B}^{t}_{p,q}(\R^2)},
	\end{align}
where
	\begin{align}\notag
		C_j:=c_j\left(\gam^{(\al-\de)/\al}2^{(\al-\de)j}+1\right),
	\end{align}
for some $(c_j)_{j\in\Z}\in\ell^q$.
\end{thm}

As a corollary of the proof of Theorem \ref{main:comm:thm}, we extend the commutator estimate of Miura (cf. \cite{miura}) to homogeneous Besov spaces.

\begin{cor}\label{comm:besov}
Suppose that $p,q$ satisfy the conditions of Theorem \ref{main:comm:thm} with $\de=0$.  Then there exists $(c_j)_{j\in\Z}\in\ell^q$ such that
	\begin{align}\notag
		\Sob{[\lpj,f]g}{L^p(\R^2)}\ls2^{-(s+t-2/p)j}c_j\Sob{{f}}{\dot{B}_{p,q}^{s}(\R^2)}\Sob{{g}}{\dot{B}^{t}_{p,q}(\R^2)}.
	\end{align}
\end{cor}

In order to prove Theorem \ref{main:comm:thm}, we view the commutator as a bilinear multiplier operator and show that its symbol satisfies 
	\begin{align}\label{hm:cond}
		\left|\bdy_\xi^{\be_1}\bdy_ \eta^{\be_2}m(\xi,\eta)\right|\ls_\be \no{\xi}^{-|\be_1|}\no{\eta}^{-|\be_2|},
	\end{align}  
In other words, we show that $m$ is of Marcinkiewicz type.  Note that condition \req{hm:cond} is weaker than that of Coifman-Meyer (cf. \cite{coif:mey}).  We must remark at this point that in general, such multipliers need not map $L^p\times L^q$ into $L^r$ \textit{for any} $1<p,q<\infty$ and $1/r=1/p+1/q$ (cf. \cite{graf:kalt}).  This can be remedied by logarithmically strengthing \req{hm:cond} as Grafakos and Kalton demonstrated in \cite{graf:kalt}.  However, the fact that we work with Besov spaces provides additional localizations, which greatly simplify the situation.

\begin{thm}\label{axis:thm}
Suppose $m:\R^d\times\R^d\goesto\R$ satisfies \req{hm:cond} for sufficiently many multi-indices $|\be|\geq0$ with $\be=\be_1+\be_2$, and that $\spt m(\xi,\ \cdotp)\sub[1/2\ls\no{\eta}\ls 2]$, for all $\xi\in\R^d$.  Then for all $1<p<\infty$, $1\leq q\leq\infty$ such that $1/r=1/p+1/q$, the associated bilinear multiplier operator $T_m:L^p(\R^d)\times L^q(\R^d)\goesto L^r(\R^d)$ satisfies
	\begin{align}\notag%\label{bounds}
		\Sob{T_m(f,g)}{L^r}\ls\Sob{f}{L^p}\Sob{g}{L^q}.
	\end{align}
\end{thm}

A prototypical example of a bilinear operator satisfying \req{hm:cond} is $T(f,g)=Hf\cdotp Hg$, where $H$ is the Hilbert transform.  Indeed, boundedness would then follow from H\"older's inequality.  The role of the smooth localization in $\eta$ in Theorem \ref{axis:thm} is that it allows us to treat the bilinear multiplier as a linear multiplier, and effectively reduce the situation to the simpler case of $Hf\cdotp Hg$.  Thus, Besov spaces provide an easy setting with which to work with bilinear Marcinkiewicz multipliers.%The fact that our multipliers are well-localized is inherited from working in Besov spaces, thus making these spaces a natural setting for bilinear Marcinkiewicz multipliers. 

The proof of Theorem \ref{axis:thm} is elementary and relies on classical techniques.  We relegate its proof to the Appendix, while the proofs of Theorems \ref{main:thm} and  \ref{main:comm:thm} can be found in Sections \ref{sect:apriori} and \ref{sect:comm}, respectively.

%SECTION: PRELIMINARY ESTIMATES
\section{Preliminary estimates}
The following heat kernel estimate was proved in \cite{miura} for $L^2$.  We extend it to $L^p$ here.
\begin{lem}\label{lem:heat:ker}
Let $2\leq p<\infty$.  Then there exist a constants $c_1,c_2>0$ such that
	\begin{align}
		e^{-c_1t2^{\kap j}}\Sob{\lpj u}{L^p}\leq\Sob{e^{-t\Lam^\kap}\lpj u}{L^p}\leq e^{-c_2t2^{\kap j}}\Sob{\lpj u}{L^p},
	\end{align}
holds for all $t>0$.
\end{lem}
\begin{proof}
Let $u_j:=e^{-t\Lam^\kap}\lpj u$.  Then $u_j$ satisfies the initial value problem
	\begin{align}\label{heat:eqn}
		\begin{cases}
			\bdy_tu_j+\Lam^\kap u_j=0\\
				u_j(x,0)=\lpj u(x).
		\end{cases}
	\end{align}
Multipying \req{heat:eqn} by $u_j|u_j|^{p-2}$ and integrating gives
	\begin{align}\notag
		\frac{1}p\frac{d}{dt}\Sob{u_j}{L^p}^p+\int (\Lam^\kap u_j)u_j|u_j|^{p-2}\ dx=0.
	\end{align}
Then by applying Lemmas \ref{gen:bern} and \ref{pos:lem}, then dividing by $\Sob{u_j}{L^p}^{p-1}$ we obtain
	\begin{align}\notag
		\frac{d}{dt}\Sob{u_j}{L^p}+c_12^{\kap j}\Sob{u_j}{L^p}\leq0,
	\end{align}
Similarly, by H\"older's inequality we obtain
	\begin{align}\notag	
		\frac{d}{dt}\Sob{u_j}{L^p}+c_22^{\kap j}\Sob{u_j}{L^p}\geq0.
	\end{align}
An application of Gronwall's inequality gives
	\begin{align}
		e^{-c_22^{\kap j}t}\Sob{u_j(0)}{L^p}\leq\Sob{u_j(t)}{L^p}\leq e^{-c_12^{\kap j}t}\Sob{u_j(0)}{L^p},
	\end{align}
which completes the proof.
\end{proof}

We will require the following estimates on the Gevrey operator.

\begin{lem}\label{lem:lin:gev}
Let $0<\al<\kap$ and $1\leq p\leq\infty$.  If $\Lam^\al f, G_\gam\Lam^\kap f\in L^p$, then
	\begin{align}\label{lin:gev}
		\Sob{G_\gam\Lam^\al \lpj f}{L^p}\ls \Sob{\Lam^\al \lpj f}{L^p}+\gam^{-(1-\kap/\al)}\Sob{G_\gam\Lam^\kap\lpj f}{L^p},
	\end{align}
for all $j\in\Z$.
\end{lem}

\begin{proof}
Fix an integer $k$, to be chosen later, such that $N:=2^{k-3}$.  Denote by $\til{\lp}_j$ the augmented operator $\lp_{j-1}+\lpj+\lp_{j+1}$.  Observe that
	\begin{align}\notag
		G_\gam\Lam^\al\lpj f=G_\gam S_{k}(\Lam^\al \lpj f)+\Lam^{-(\kap-\al)}(I-S_{k}){\lpj}(G_\gam\Lam^\kap \til{\lp}_j{f}).
	\end{align}
Observe that $G_\gam S_{k}\in L^1$.  Indeed, by Lemma \ref{bern} we have
	\begin{align}\label{lin:gev1}
		\Sob{G_\gam S_{k}}{L^1}&\leq\sum_{n=0}^\infty\frac{\lam^n\gam^n}{n!}\Sob{\Lam^{\al n}S_{k}}{L^1}\leq e^{c\gam 2^{k\al}},
	\end{align}
for some absolute constant $c>0$.  On the other hand, observe that $\check{m}:=\Lam^{-(\kap-\al)}(I-S_{k})\lpj\in\Sch$.  Let $g:=G_\gam\Lam^\kap \til{\lp}_j{f}$.  We consider three cases.

If $2^{j+2}\leq N$, then $g\equiv0$.  If $N\leq 2^{j-2}$, then Lemma \ref{bern} and Young's convolution inequality implies that
	\begin{align}\notag
		\Sob{T_mg}{L^1}\ls 2^{-(\kap-\al)j}\ls N^{-(\kap-\al)},
	\end{align}
where $T_m$ is convolution with $\check{m}$.  Similarly, if $2^{j-1}\leq N\leq 2^{j+1}$, then
	\begin{align}\label{lin:gev2}
		\Sob{T_mg}{L^1}\ls N^{-(\kap-\al)}.
	\end{align}
Therefore, for any $N>0$
	\begin{align}\notag
		\Sob{G_\gam\Lam^\al\lpj f}{L^p}\ls e^{\gam N^\al}\Sob{\Lam^\al\lpj f}{L^p}+N^{-(\kap-\al)}\Sob{G_\gam\Lam^\kap\til{\lp}_j f}{L^p}.
	\end{align}
Finally, choose $k:=[\al^{-1}\log_2(1/\gam)]$, where $[x]$ denotes the greatest integer $\geq x$.  Then $N\sim\gam^{-1/\al}$, which gives \req{lin:gev}.
\end{proof}

Finally, we will require the following estimate for the solution to the linear heat equation \req{heat:eqn}.

\begin{lem}\label{lin:heat:eqn}
Let $\al<\kap$, $\s>0$, and $\be\geq0$ and suppose that $\tht_0\in\Bessig(\R^2)$.  Then for any $T\geq0$ 
	\begin{enumerate}
		\item[(i)] $\Sob{e^{-( \cdotp) \Lam^\kap}\tht_0}{X_T}\ls\Sob{\tht_0}{\Bessig}$,
		\item[(ii)] $\lim_{T\goesto0}\Sob{\tht_0}{X_T}=0.$
	\end{enumerate}
\end{lem}

\begin{proof}
Observe that for $b<1$, we have $e^{ax^b-cx}\leq1$ for $x>1$ and $e^{ax^b-cx}\ls e^{-cx}$ for $0\leq x\leq1$.  Arguing as in Lemma \ref{lem:lin:gev} and applying Lemma \ref{lin:heat:eqn} we get
	\begin{align}\notag
		\Sob{\e e^{-t\Lam^\kap}\lpj \tht_0}{L^p}\ls e^{c_1\lam t^{\al/\kap}2^{j\al}-c_2t2^{j\kap}}\Sob{\lpj\tht_0}{L^p}\ls\Sob{e^{-c_3t\Lam^\kap}\lpj\tht_0}{L^p}.
	\end{align}
for some $c_1,c_2,c_3>0$.  Then by Lemma \ref{lem:heat:ker} we have
	\begin{align}\label{sch:case}
		\notag\Sob{\e e^{-t\Lam^\kap}\tht_0}{\Bess}^q&=\sum_j2^{(\s+\be)jq}\Sob{ e^{\lam t^{\al/\kap}\Lam^\al-t\Lam^\kap}\lpj \tht_0}{L^p}^q\\
			\notag&\ls \sum_j2^{\be jq}e^{-2qc_3t2^{j\kap}}\left(2^{\s j}\Sob{\lpj \tht_0}{L^p}\right)^q\\
			&\ls t^{-(\be q)/\kap}\Sob{\tht_0}{\Bessig}^q.
	\end{align}
This proves $(i)$.  Now we prove $(ii)$.  Then for let $\eps>0$, there exists $\tht_0^{\eps}\in\Sch$ such that $\Sob{\tht_0-\tht_0^\eps}{\Bessig}<\eps$.  In particular, $\tht_0^\eps\in\Bess$.  Observe that for $0<t\leq T$
	\begin{align}
		\notag\Sob{e^{-t\Lam^\kap}\til{\tht}_0}{\Bess}&\ls\Sob{e^{-t\Lam^\kap}\til{\tht}_0^\eps}{\Bess}+\Sob{e^{-t\Lam^\kap}\til{\tht}_0- e^{-t\Lam^\kap}\til{\tht}_0^\eps}{\Bess}\\
			\notag&\ls\Sob{\tht_0^\eps}{\Bess}+\Sob{e^{-c_3t\Lam^\kap}(\tht_0-\tht_0^\eps)}{\Bess}\\
			\notag&\ls t^{-\be/\kap}\left(T^{\be/\kap}\Sob{\tht_0^\eps}{\Bess}\right)+t^{-\be/\kap}\Sob{\tht_0-\tht_0^\eps}{\Bessig},
	\end{align}
where we have applied \req{sch:case} to $\tht_0-\tht_0^\eps$.  This implies $(ii)$ and we are done.
\end{proof}

%SECTION: COMMUTATOR ESTIMATES
\section{Commutator estimates}\label{sect:comm}
In this section, we derive estimates for the commutator
	\begin{align}\label{comm}
		[G_{\gam}\lpj,f]g:=G_{\gam}\lpj(fg)-fG_{\gam}\lpj g,
	\end{align}
where $G_{\gam}:=e^{\gam\Lam^\al}$ and $0<\al<\kap\leq1$, where $\kap$ is the order of dissipation in \req{qg}.  For convenience, we will often use the notation $\til{f}:=G_\gam f$.  Recall that we want to prove the following statement.

\begin{thmcomm}\label{thm:comm:est}
Let $1< p<\infty$ and $1\leq q\leq\infty$.  Let $\gam,\de>0$ and $N>1$ such that $\de<1$.  Suppose the following holds
	\begin{enumerate}
		\item[(i)]$2/p<s<1+2/p-\de$,
		\item[(ii)] $t<2/p$,
		\item[(iii)] $s+t>2/p$.
	\end{enumerate}
Then there exists $(C_j)_{j\in\Z}\in\ell^q$ such that
	\begin{align}
		\Sob{[G_{\gam}\lpj,f]g}{L^p}\ls2^{-(s+t-2/p)j}C_j\left(\gam^{(\al-\de)/\al}2^{(\al-\de)j}+1\right)\Sob{\til{f}}{\dot{B}_{p,q}^{s}}\Sob{\til{g}}{\dot{B}^{t}_{p,q}}.
	\end{align}
\end{thmcomm}

To prove this, we will require the Fa\`a di Bruno formula, whose statement we recall from \cite{bcd} for convenience.
\begin{lem}[Fa\`a di Bruno formula]\label{fdb}
Let $u:\R^d\goesto\R^m$ and $F:\R^m\goesto\R$ be smooth functions.  For each multi-index $\al\in\N^d$, we have
	\begin{align}
		\bdy^\al(F\circ u)=\sum_{\mu,\nu}C_{\mu,\nu}\bdy^\mu F\prod_{\substack{1\leq|\be|\leq|\al|\\ 1\leq j\leq m}}(\bdy^\be u^j)^{\nu_{\be_j}},
	\end{align}
where the coefficients $C_{\mu,\nu}$ are nonnegative integers, and the sum is taken over those $\mu$ and $\nu$ such that $1\leq|\mu|,|\nu|\leq|\al|$, $\nu_{\be_j}\in\N^*$,
	\begin{align}
		\sum_{1\leq|\be|\leq|\al|}\nu_{\be_j}=\mu_j,\ \text{for}\  1\leq j\leq m, \ \ \text{and}\ \sum_{\substack{1\leq|\be|\leq|\al|\\1\leq j\leq m}}\be\nu_{\be_j}=\al.
	\end{align}
\end{lem}

We will repeatedly apply this formula to the function
	\begin{align*}
		(F\circ u)(\xi,\eta)=e^{\gam R_\al(\xi,\eta)},
	\end{align*}
where 
	\begin{align*}
		R_\al(\xi,\eta):&=\no{\xi+\eta}^\al-\no{\xi}^\al-\no{\eta}^\al.
	\end{align*}
For convenience, we provide that application here.  By Lemma \ref{fdb} we have
		\begin{align}\label{dexp}
			\bdy^\be(F\circ u)(\xi,\eta)=\sum_{\mu,\nu}C_{\mu,\nu} \gam^{ |\mu|}e^{\gam R_\al(\xi,\eta)}\prod_{1\leq |b|\leq|\be|}(\bdy^bR_\al(\xi,\eta))^{\nu_b}
		\end{align}
for all  $\be\in\N^2$, where $\nu=(\nu_1,\nu_2)$, $1\leq |\mu|\leq|\be|$ and 
		\begin{align}\label{fdb:indices}
			\sum_{1\leq|b|\leq|\be|}\nu_{b}=\mu\ \ \ \text{and}\ \sum_{1\leq |b|\leq|\be|}b\nu_b=\be.
		\end{align}
In order to apply Theorem \ref{axis:thm}, we will require $R_\al$ to satisfy certain derivative estimates.

\begin{prop}\label{R:est}
Let $0<\al\leq1$, $\s\in[0,1]$, and define $R_{\al,\s}:\R^2\times\R^2\goesto\R$ by
	\begin{align}\notag
		R_{\al,\s}(\xi,\eta):=\no{\xi+\eta\s}^\al-\no{\xi}^\al-\no{\eta}^\al.
	\end{align}
Suppose that $\ell+3\leq k$ and $2^{k-1}\leq\no{\xi}\leq2^{k+1}$ and $2^{\ell-1}\leq\no{\eta}\leq2^{\ell+1}$.  Then
	\begin{align}\label{R:infty}
		\left|R_{\al,\s}(\xi,\eta)\right|\ls 2^{\ell\al},
	\end{align}
and
	\begin{align}\label{R:deriv}
		\left|\bdy_{\xi}^{\be_1}\bdy_{\eta}^{\be_2}R_{\al,\s}(\xi,\eta)\right|\ls 2^{\ell\al}\no{\xi}^{-|\be_1|}\no{\eta}^{-|\be_2|},
	\end{align}
for all multi-indices $\be_1,\be_2\in\N^2$.

If $j+3\leq k$ with $2^{j-1}\leq\no{\eta}\leq2^{j+1}$ and $2^{k-1}\leq\no{\xi},\no{\xi+\eta}\leq2^{k+1}$, then
	\begin{align}\label{R:deriv2}
		\left|\bdy_{\xi}^{\be_1}\bdy_{\eta}^{\be_2}R_{\al,1}(\xi,-\xi-\eta)\right|\ls2^{k\al}\no{\xi}^{-|\be_1|}\no{\eta}^{-|\be_2|},
	\end{align}
for all $\be_1,\be_2\in\N^d$.
\end{prop}

\begin{proof}
We prove \req{R:infty} and \req{R:deriv}.  The proof of \req{R:deriv2} is easier.  %Indeed, one need only observe that $\no{\xi+\eta\tau}\sim2^k$ for all $0\leq\tau\leq1$.  %er cases $\tau\leq1/2$ and $\tau\geq1/2$.
%We develop the power series of $e^{\gam(\no{\xi+\eta}^\al-\no{\xi}^\al-\no{\eta}^\al)}$.  Indeed
%	\begin{align}
%		m_\ell(\xi,\eta)=\sum_{n\geq0}\frac{1}{n!}\gam^nR_\al(\xi,\eta)^n,
%	\end{align}
%where
%	\begin{align}
%		R_\al(\xi,\eta):=\no{\xi+\eta}^\al-\no{\xi}^\al-\no{\eta}^\al.
%	\end{align}

To proof of \req{R:infty} follows from the triangle inequality
	\begin{align}\notag
		\left|R_{\al,\s}(\xi,\eta)\right|\ls(1-\s)\no{\eta}^\al\ls2^{\ell\al}
	\end{align}
To prove \req{R:deriv}, we first apply the mean value theorem to write
	\begin{align}\notag
		R_{\al,\s}(\xi,\eta)=\int_0^1\no{\xi+\eta\tau\s}^{\al-2}((\xi\cdotp\eta)\s+\no{\eta}^2\s^2\tau)\ d\tau-\no{\eta}^\al.
	\end{align}
%Now, by the Fa\`a di Bruno formula we get
%	\begin{align}
%		\bdy^\be R_\al(\xi,\eta)^n=\sum_{\mu,\nu}c_{\mu,\nu}\frac{n!}{(n-|\mu|)!}R_\al(\xi,\eta)^{n-|\mu|}\prod_{1\leq|b|\leq|\be|}(\bdy^bR_\al(\xi,\eta))^{\nu_b}.
%	\end{align}
Let $\be\in\N^2\times\N^2$, where $\be=(\be_1,\be_2)$.  We need to consider three cases.  First suppose that $\be_1=0,\be_2\neq0$.  Then observe that 
	\begin{align}\notag
		\left|\bdy^{\be}R_\al(\xi,\eta)\right|&\ls\sum_{\be}c_\be\int_0^1\left(\no{\xi+\eta\s\tau}^{\al-2-|\be_1|}\bdy^{\be_2}((\xi\cdotp\eta)\s+\no{\eta}^{2}\s^2\tau)\right)\ d\tau+\no{\eta}^{\al-|\be|}.
					%&\ls \sum_{b=b_1+b_2+b_3}c_b2^{\ell(\al-|b|)}\left(2^{(k-\ell)(\al-|b|)}+1\right)\ls2^{-\ell(|b|-\al)}.
	\end{align}
Observe that
	\begin{align}\notag
		\left|\bdy^{\be_2}((\xi\cdotp\eta)\s+\no{\eta}^2\s^2\tau)\right|\ls\begin{cases}
							2^{k+\ell}&,|\be_2|=0\\
							2^k&,|\be_2|=1\\
							1&,|\be_2|=2\\
							0&,|\be_2|\geq3.
				\end{cases}
	\end{align}
In each case, using the fact that $\al<1$ and $\ell+3\leq k$, we have %$2^{k(\al-2-|b_1|)}2^{k+\ell}+2^{\ell(\al-|b|)}\ls2^{\ell(\al-|b_1|)}$
	\begin{align}\notag
		\left|\bdy^{\be}R_{\al,\s}(\xi,\eta)\right|&\ls2^{\ell\al}2^{-\ell|\be|},
	\end{align}
which implies \req{R:deriv} since $\be=(\be_1,0)$.

If $\be_1\neq0,\be_2=0$, then
	\begin{align}\notag
		\left|\bdy^{\be_2}((\xi\cdotp\eta)\s+\no{\eta}^2\s^2\tau)\right|\ls\begin{cases}
							2^{k+\ell}&,|\be_2|=0\\
							2^\ell&,|\be_2|=1\\
							0&,|\be_2|\geq2.
				\end{cases}
	\end{align}
Now in each case we have 
	\begin{align}\notag
		\left|\bdy^\be R_{\al,\s}(\xi,\eta)\right|&\ls2^{k(\al-1-|\be|)}2^{\ell}\ls2^{\ell\al}2^{-k|\be|},
	\end{align}
which again implies \req{R:deriv}.

 Finally, if $\be_1\neq0,\be_2\neq0$, then
	\begin{align}\notag
		\left|\bdy^{\be_2}((\xi\cdotp\eta)\s+\no{\eta}^2\s^2\tau)\right|\ls\begin{cases}
							2^{k+\ell}&,|\be_2|=0\\
							2^k&,|\be_2|=1, \be_2^\xi=0\\
							2^\ell&,|\be_2|=1,\be_2^\eta=0\\
							1&,|\be_2|=2,\be_2^\xi=0\ \text{or}\ |\be_2^\xi|=|\be_2^\eta|=1\\
							0&,|\be_2|\geq3\ \text{or}\ |\be_2|=2, \be_2^\eta=0.
				\end{cases}
	\end{align}
and arguing as before, we obtain
	\begin{align}
		\left|\bdy^{\be}R_{\s,\al}(\xi,\eta)\right|&\ls2^{\ell\al}2^{-k|\be_\xi|}2^{-\ell|\be_\eta|},
	\end{align}
where $\be=(\be_\xi,\be_\eta)=(\be_1,\be_2)\in\N^4\times\N^4$, $\be_i=(\be_i^\xi,\be_i^\eta)\in\N^2\times\N^2$, and  $\be_\xi=\be_1^\xi+\be_2^\xi$ and $\be_\eta=\be_1^\eta+\be_2^\eta$.  Thus, \req{R:deriv} is again established.
\end{proof}

We will also need the following ``rotation" lemma.

\begin{lem}\label{rot:lem}
Let $T_m$ be a bilinear multiplier operator with multiplier $m:\R^d\times\R^d\goesto\R$.  Then there exists a bounded multiplier $\til{m}$ such that
	\begin{align}
		\lb T_m(f,g),h\rb=\lb T_{\til{m}}(h,g), f\rb,
	\end{align}
for all $f,g,h\in\Sch(\R^d)$.  In particular, if $T_m:L^p\times L^q\goesto L^r$ is bounded for $1/r=1/p+1/q$, then $T_{\til{m}}:L^{r'}\times L^q\goesto L^p$ is bounded, where $r'$ is the H\"older conjugate of $r$.  Moreover, $|\bdy^\be m|=|\bdy^\be\til{m}|$ for all $\be\in\N^d$.
\end{lem}

\begin{proof}
By change of variables we have
	\begin{align}
		\int T_m(f,g)(x){h(x)}\ dx&=c_d^2\int\int\int e^{ix\cdotp(\xi+\eta)}m(\xi,\eta)\hat{f}(\xi)\hat{g}(\eta){h(x)}\ d\xi\ d\eta\ dx\notag\\
			\notag&=(-1)^dc_d^2\int\int\int e^{-ix\cdotp\nu}m(\xi,\xi-\nu)\hat{f}(\xi)\hat{g}(-\nu-\xi){h(x)}\ dx\ d\xi\ d\nu\\
			\notag&=(-1)^dc_d\int\int m(\xi,-\nu-\xi)\hat{g}(-\xi-\nu)\hat{{h}}(\nu)\hat{f}(\xi)\ d\nu\ d\xi\\
			\notag&=(-1)^dc_d^2\int\int\int e^{-ix\cdotp\xi}m(\xi,-\xi-\nu)\hat{g}(-\nu-\xi)\hat{{h}}(\nu)f(x)\ d\nu\ d\xi\ dx\\
			%\notag&=(-1)^d{\int\int\int e^{-ix\cdotp\xi}{m}(\xi,-\xi-\nu)\hat{h}(\nu)\hat{{g}}(-\nu-\xi){f(x)}\ d\nu\ d\xi\ dx}\\
			\notag&={\lb T_{\til{m}}(h,{g}),{f}\rb},
	\end{align}
where
	\begin{align}
		\til{m}(\xi,\eta):=(-1)^d{m}(\xi,-\xi-\eta).
	\end{align}
Then obviously, $\til{m}$, $|\bdy^\be \til{m}|=|\bdy^\be m|$.  Boundedness of $T_{\til{m}}$ then follows from duality.
\end{proof}

Before we proceed to the proof of Theorem \ref{thm:comm:est}, observe that the paraproduct decomposition yields
	\begin{align}\label{pp:comm}
		[G_{\gam}\lpj,f]g=&\sum_kG_{\gam}\lpj(S_kf\lpk g)+G_{\gam}\lpj(\lpk fS_kg)+G_{\gam}\lpj(\til{\lp}_k f\lpk g)\\
				&-\left(\sum_k(S_kf)(\lpj\lpk \til{g})+(\lpk f)(\lpj S_k\til{g})+(\til{\lp}_ kf)(\lpj\lpk \til{g})\right).\notag
	\end{align}
Then by the localization properties in \req{lp:loc}, we can reduce \req{pp:comm} to
	\begin{align}\label{pp:comm:cases}
		[G_{\gam}\lpj,f]g=&\sum_{|k-j|\leq4}[G_{\gam}\lpj,S_kf]\lpk g+G_\gam\lpj(\lpk fS_kg)\\
				&+\sum_{k\geq j+3}G_{\gam}\lpj(\til{\lp}_k f\lpk g)\notag\\
				&-\sum_{k\geq j+1}\lpk f\lpj S_k\til{g}-\sum_{|k-j|\leq2}\til{\lp}_k f\lpj\lpk\til{g}.\notag
	\end{align}

Now we prove Theorem \ref{thm:comm:est} by considering the cases as they are presented in each row of \req{pp:comm:cases}.

\subsection{Case: $k\geq j+3$}
%It suffices to consider the term $G_{\gam}\lpj(\til{\lp}_k f\lpk g)$ since the other terms in this case can be reduced to a term of this form.  Indeed
%	\begin{align}\notag
%		2^{k-3}\leq|\eta|-|\xi+\eta|\leq|\xi|\leq|\xi+\eta|+|\eta|\leq2^{k+3}.
%	\end{align}
%Therefore, $G_{\gam}\lpj(\til{\lp}_k f\lpk g)$ can replace $G_{\gam}\lpj(S_kf\lpk g)$ and $G_\gam\lpj(\lpk f\til{\lp}_kg)$ can replace $G_\gam\lpj(\lpk f S_k g)$, so that these cases can be treated similarly.
%\footnote{Note the slight abuse in notation; strictly speaking $S_k$ should be replaced by $\lp_{k-1}+\lpk+\lp_{k+1}$.  This ambiguity does not affect the proof in a significant manner}.

First, we rewrite $G_{\gam}\lpj(\til{\lp}_k f\lpk g)$ as
	\begin{align}\label{kgtrj:insertgev}
		G_{\gam}\lpj(G_\gam^{-1}\til{\lp}_k \til{f}G^{-1}_\gam\lpk \til{g})
	\end{align}
The multiplier associated to \req{kgtrj:insertgev} is
	\begin{align}\label{mult:kgtrj}
		m_{k,j}(\xi,\eta):=e^{\gam(\no{\xi+\eta}^\al-\no{\xi}^\al-\no{\eta}^\al)}\ph_j(\xi+\eta)\til{\ph}_k(\xi)\ph_k(\eta),
	\end{align}
where $\til{\ph}_k=\sum_{|k-\ell|\leq 2}\ph_\ell$.  By Lemma \ref{rot:lem}, it suffices to prove 
	\begin{align}\label{claim:kgtrj}
		|\bdy_\xi^{\be_1}\bdy_\eta^{\be_2} m_{k,j}(\xi,-\xi-\eta)|\ls\no{\xi}^{-|\be_1|}\no{\eta}^{-|\be_2|}.
	\end{align}
We can then apply Theorem \ref{axis:thm}.

Observe that for $\be=(\be_1,\be_2)$, by \req{dexp} and \req{fdb:indices} we have
	\begin{align}\notag
		|\bdy^\be &m_{k,j}(\xi,-\xi-\eta)|\\
			\notag&\ls\sum_{\mu,\nu}C_{\mu,\nu} \gam^{ |\mu|}e^{\gam(\no{\eta}^\al-\no{\xi}^\al-\no{\xi+\eta}^\al)}\prod_{1\leq |b|\leq|\be|}(2^{k(\al-|b|)})^{\nu_b}\\
			\notag&\ls2^{-k|\be|}\sum_{\mu,\nu}C_{\mu,\nu}(\gam2^{k\al})^{ |\mu|}e^{\gam(\no{\eta}^\al-\no{\xi}^\al-\no{\xi+\eta}^\al)}
	\end{align}
where we have used \req{R:deriv2} in Proposition \ref{R:est}.  Also, by the triangle inequality%Proposition \ref{concavity}, we have%Since $k\geq j+4$, these facts imply that
	\begin{align}\notag
		\no{\eta}^\al-\no{\xi}^\al-\no{\xi+\eta}^\al\ls -c_\al2^{k\al}.%(1-2^{(k-j-2)\al+1})\ls-2^{j\al}.
	\end{align}
for some absolute constant $c_\al>0$.  Thus
	\begin{align}\label{kgtrj:const}
		|\bdy^\be &m_{k,j}(\xi,-\xi-\eta)|&\ls2^{-k|\be|}\sum_{\mu,\nu}C_{\mu,\nu}(\gam2^{k\al})^{|\mu|}e^{-c_\al\gam2^{k\al}}\ls 2^{-k|\be_1|}2^{-k|\be_2|}
	\end{align}
holds for all $\xi\in\R^2$, which implies \req{claim:kgtrj}.

Therefore by Theorem \ref{axis:thm}, we have
	\begin{align}\label{Lp:bounds:kgtrj}
		\Sob{G_{\gam}\lpj(\til{\lp}_k f\lpk g)}{L^r}\ls\Sob{\til{\lp}_k \til{f}}{L^p}\Sob{\lpk \til{g}}{L^q},
	\end{align}
where $1/r=1/p+1/q$ and $1< r,p<\infty$, $1\leq q\leq \infty$.  

Now let $\s=s+t-2/p$.  By \req{Lp:bounds:kgtrj} and the Bernstein inequalities
	\begin{align}\notag
		2^{\s j}&\Sob{G_{\gam}\lpj(\lpk f\lpk g)}{L^p}\\
		&\ls\sum_k\underbrace{\chi_{[n\geq3]}(k-j)}_{\mu_{k-j}}\underbrace{2^{-(s+t-2/p)(k-j)}}_{a_{k-j}}\underbrace{2^{s k}\no{\til{\lp}_k \til{f}}_{L^p}}_{b_k}\underbrace{2^{tk}\no{\lpk \til{g}}_{L^p}}_{c_k}.\notag
	\end{align}
Observe that by Young's convolution inequality
	\begin{align}\notag
		\left(\sum_{k}(\mu_{k-j}a_{k-j}b_k)^q\right)^{1/q}\leq\left(\sum_{k\geq3}a_k\right)\left(\sum_{k}b_k^q\right)^{1/q},
	\end{align}
which will be finite provided that
	\begin{align}\notag
		s+t-2/p>0.
	\end{align}
Thus
	\begin{align}\label{kgtrj:diag}
		2^{(s+t-2/p) j}\sum_{k\geq j+3}\Sob{G_{\gam}\lpj(\til{\lp}_k f\lpk g)}{L^r}&\ls c_j\Sob{\til{f}}{\dot{B}_{p,q}^{s}}\Sob{\til{g}}{\dot{B}^{t}_{p,\infty}},
	\end{align}
where 
	\begin{align}\notag
		c_j:=\Sob{\til{f}}{\dot{B}_{p,q}^{s}}^{-1}\sum_k\mu_{k-j}a_{k-j}{b_k}.
	\end{align}
Observe that $(c_j)_{j\in\Z}\in\ell^q$.

\subsection{Cases: $k\geq j+1$ and $|k-j|\leq1$}
Then the corresponding terms are $\lpk f\lpj S_k\til{g}$ and $\lpk f\lpj\lpk \til{g}$, respectively.  By H\"older's inequality and Bernstein we have
	\begin{align}
		2^{\s j}\Sob{\lpk f\lpj S_k\til{g}}{L^p}\ls c_j2^{-(s-2/p)(k-j)}2^{s k}\Sob{\lpk{f}}{L^p}\Sob{\til{g}}{\dot{B}^{t}_{p,q}}.
	\end{align}
where 
	\begin{align}\notag
		c_j:=\Sob{\til{g}}{\dot{B}^{t}_{p,q}}^{-1}2^{tj}\Sob{\lpj\til{g}}{L^p}.
	\end{align}
Observe that by H\"older's  inequality
	\begin{align}\notag
		\sum_{k\geq j+1}2^{-(s-2/p)(k-j)}\chi_{[n\geq1]}(k-j)2^{s k}\Sob{\lpk{f}}{L^p}\ls\left(\sum_{k\geq1}2^{-(s-2/p)kq'}\right)^{1/q'}\Sob{{f}}{\dot{B}_{p,q}^s},
	\end{align}
which is finite provided that $s-2/p>0$.  Therefore %use maximal function to get s\geq 2/p?????
	\begin{align}\label{kgtrj:holder}
		2^{(s+t-2/p)j}\sum_{k\geq j+1}\Sob{\lpk f\lpj S_k\til{g}}{L^p}\ls c_j\Sob{{f}}{\dot{B}_{p,q}^s}\Sob{\til{g}}{\dot{B}^{t}_{p,q}}
	\end{align}

Similarly, we have for any $s\in\R$
	\begin{align}\label{ksimj:holder}
		2^{(s+t-2/p)j}\sum_{|k-j|\leq1}\Sob{\til{\lp}_k f\lpk \til{g}}{L^p}\ls c_j\Sob{{f}}{\dot{B}_{p,q}^s}\Sob{\til{g}}{\dot{B}^{t}_{p,q}},
	\end{align}
where
	\begin{align}\notag
		c_j:=\Sob{\til{g}}{\dot{B}^{t}_{p,q}}^{-1}2^{tj}\Sob{\lpj\til{g}}{L^p}.
	\end{align}

\subsection{Case: $|k-j|\leq4$}

\subsubsection{Non-commutator terms}

Here we treat $G_\gam\lpj(\lpk fS_kg)$.  First, rewrite $G_\gam\lpj(\lpk fS_kg)$ as
	\begin{align}\label{ksimj1:insertgev}
		\sum_{\ell\leq k-3}G_{\gam}\lpj(G_\gam^{-1}\lpk \til{f}G_\gam^{-1}\lpl \til{g}).
	\end{align}
We claim that the associated multiplier satisfies the following bounds
	\begin{align}\label{claim1:ksimj}
		|\bdy_{\xi}^{\be_1}\bdy_\eta^{\be_2} m_{j,k}(\xi,\eta)|\ls\no{\xi}^{-|\be_1|}\no{\eta}^{-|\be_2|},
	\end{align} 
where
	\begin{align}\label{mult:ksimj1}
		m_{j,k}(\xi,\eta)=e^{\gam(\no{\xi+\eta}^\al-\no{\xi}^\al-\no{\eta}^\al)}\ph_j(\xi+\eta)\ph_k(\xi)\ph_k(\eta).
	\end{align}

To this end, let $\be=(\be_\xi,\be_\eta)$ and observe that by \req{dexp} and Proposition \ref{R:est} we have
	\begin{align}\notag
		|\bdy^\be e^{\gam R_\al(\xi,\eta)}|&\ls\sum_{\mu,\nu}C_{\mu,\nu} \gam^{ |\mu|}e^{\gam R_\al(\xi,\eta)}\prod_{1\leq |b|\leq|\be|}(2^{\ell(\al-|b_\eta|)}2^{-k|b_\xi|})^{\nu_b}.
	\end{align}
%Observe that
%	\begin{align}\notag
%		\no{\xi+\eta}^\al-\no{\xi}^\al-\no{\eta}^\al=\no{\eta}(\no{\xi/\no{\eta}+\eta/\no{\eta}}^\al-(\no{\xi}/\no{\eta})^\al-1)
%	\end{align}
Since $\no{\eta}\sim2^{\ell}$ and $k-\ell\geq3$, it follows by Lemma \ref{concavity} that
	\begin{align}\notag
		\no{\xi+\eta}^\al-\no{\xi}^\al-\no{\eta}^\al\ls -c_\al2^{\ell\al}.
	\end{align}
Thus, by \req{fdb:indices} we get
	\begin{align}\label{ksimj:const1}
		|\bdy^\be e^{\gam R_\al(\xi,\eta)}|&\ls \sum_{\mu,\nu} C_{\mu,\nu}\gam^{|\mu|}e^{-c_\al\gam2^{\ell\al}}2^{\ell(\al|\mu|-|\be_\eta|)}2^{-k|\be_\xi|}\\
			\notag&\ls2^{-k|\be_\xi|}2^{-\ell|\be_\eta|}\sum_{\mu,\nu}C_{\mu,\nu}(\gam2^{\ell\al})^{|\mu|}e^{-c_\al\gam2^{\ell\al}}\\
			\notag&\ls2^{-k|\be_\xi|}2^{-\ell|\be_\eta|}%\no{\xi}^{-|\be_\xi|}\no{\eta}^{-|\be_\eta|},
	\end{align}
holds for all $\xi\in\R^2$

Hence, by the product rule and the fact that $2^k\sim2^j$, we can conclude that
	\begin{align}\notag
		|\bdy_\xi^{\be_1}\bdy_\eta^{\be_2} m_{j,k}(\xi,\eta)|\ls 2^{-k|\be_1|}2^{-\ell|\be_2|}.
	\end{align}
for all $\xi\in\R^2$, which implies \req{claim1:ksimj}.

Therefore by Theorem \ref{axis:thm}, we have
	\begin{align}\label{Lp:bounds:ksimj}
		\Sob{G_{\gam}\lpj(\lpk fS_{k} g)}{L^r}\ls\sum_{\ell\leq k-1}\Sob{\lpk \til{f}}{L^p}\Sob{\lpl \til{g}}{L^q},
	\end{align}
where $1/r=1/p+1/q$ and $1\leq r<\infty$, $1< p<\infty$, $1<q\leq\infty$. 
	
Now let $\s=s+t-2/p$ and $N>1$.  Let $p^*:=(pN)/(N-1)$.  Then by \req{Lp:bounds:ksimj}, the Bernstein inequalities, and the fact that $|k-j|\leq3$, we have
	\begin{align}\notag
		2^{\s j}&\Sob{G_{\gam}\lpj(\lpk fS_{k+1} g)}{L^p}\\
			\notag&\ls\sum_{\ell\leq k-1}2^{(\s-s-t+2/p^*)k}2^{sk}\Sob{\lpk\til{f}}{L^{p^*}}2^{t\ell}\Sob{\lpl\til{g}}{L^{p}}2^{-(2/p^*-t)(k-\ell)}\\
			\notag&\ls2^{sj}\Sob{\lpj\til{f}}{L^{p}}\sum_{\ell\leq k-1}2^{t\ell}\Sob{\lpl\til{g}}{L^{p}}2^{-(2/p^*-t)(k-\ell)}
	\end{align}
Let $t<2/p$.  Observe that for $N$ large enough, we have $t<2/p^*$.  Then
	\begin{align}\label{ksimj:nocomm}
		2^{(s+t-2/p) j}\Sob{G_{\gam}\lpj(\lpk fS_{k+1} g)}{L^p}\ls C_j\Sob{\til{f}}{\dot{B}_{p,\infty}^s}\Sob{\til{g}}{\dot{B}_{p,q}^t},
	\end{align}
where
	\begin{align}\notag
		C_j:=\sum_{\ell\leq j+2}2^{t\ell}\Sob{\lpl\til{g}}{L^p}2^{(t-2/p^*)(j-\ell)},
	\end{align}
which satisfies $(C_j)_{j\in\Z}\in\ell^q$.

\subsubsection{Commutator term}
Finally, we consider the term $[G_{\gam}\lpj,S_kf]\lpk g$, which we will denote as $T_{m_{j,k}}(S_kf,\lpk g)$.  Observe that
	\begin{align}\notag
		T_{m_{j,k}}&(S_kf,\lpk g)(x)\\
		\notag	&=\int\int e^{ix\cdotp(\xi+\eta)}\left[G_\gam(\xi+\eta)\ph_j(\xi+\eta)-G_\gam(\eta)\ph_j(\eta)\right]\psi_k(\xi)\ph_k(\eta)\hat{f}(\xi)\hat{g}(\eta)\ d\xi d\eta.
	\end{align}
Then by the mean value theorem 
	\begin{align}\notag
		T_{m_{j,k}}(S_kf,\lpk g)(x)=\sum_{i=1,2}\sum_{\ell\leq k-3}\int_0^1 T_{m_{i,j,k,\ell,\s}}(\lpl\bdy_i \til{f},\lpk \til{g})(x)\ d\s,
	\end{align}
where
	\begin{align}
	\notag	m_{i,j,k,\ell,\s}(\xi,\eta)=&m_A(\xi,\eta)+m_B(\xi,\eta),
	\end{align}
and%\footnote{Note again our abuse of notation.}
	\begin{align}
		\notag m_A(\xi,\eta)&:=\al\gam e^{\gam R_{\al,\s}(\xi,\eta)}\no{\xi\s+\eta}^{\al-2}({\xi_i\s+\eta_i})\ph_j(\xi\s+\eta)\ph_\ell(\xi)\ph_k(\eta)\\
		\notag m_B(\xi,\eta)&:=e^{\gam R_{\al,\s}(\xi,\eta)}(\bdy_i\ph_0)(2^{-j}(\xi\s+\eta))2^{-j}\ph_\ell(\xi)\ph_k(\eta).
	\end{align}
Now observe that since $\no{\xi}\sim2^\ell$, $\no{\eta}\sim2^k$, and $k-\ell\geq3$, by Lemma \ref{concavity} there exists a constant $c_\al>0$ such that
	\begin{align}\label{tausim1}
		\no{\xi\s+\eta}-\no{\xi\s}^\al-\no{\eta}^\al\leq-c_\al\no{\xi}^\al, \ \text{for}\ \s\geq1/2,
	\end{align}
and by the triangle inequality
	\begin{align}\label{tausim0}
		\no{\xi\s+\eta}-\no{\xi}^\al-\no{\eta}^\al\leq-c_\al\no{\xi}^\al, \ \text{for}\ \s\leq1/2.
	\end{align}

%First, we estimate $m_A$.  %Fix $\de>0$.  We modify the $m_{A}$ by ${m}_{A,\de}$ so that
%	\begin{align}\notag
%		 T_{m_{A}}(\lpl\bdy_i f,\lpk g)=T_{{m}_{A,\de}}(\lpl\bdy_i f,\lpk g),
%	\end{align}
%where
%	\begin{align}\notag
%		\til{m}_{A,\de}(\xi,\eta)=\no{\xi}^{-\de}m_{A}(\xi,\eta).
%	\end{align}\
%To estimate $m_{A,\de}$, first suppose that $\tau\geq1/2$. 
Suppose that $\s\leq1/2$ and observe that by Proposition \ref{R:est}, Fa\`a di Bruno, and \req{tausim0}, we have
	\begin{align}
	\notag|\bdy^\be e^{\gam R_{\al,\s}(\xi,\eta)}|&\ls\sum_{\mu,\nu}C_{\mu,\nu} \gam^{ |\mu|}e^{\gam R_{\al,\s}(\xi,\eta)}\prod_{1\leq |b|\leq|\be|}(2^{\ell(\al-|b_\xi|)}2^{-k|b_\eta|})^{\nu_b}\\
				\notag&\ls2^{-\ell|\be_1|}2^{-k|\be_2|}\sum_{\mu,\nu}C_{\mu,\nu}(\gam2^{\ell\al})^{|\mu|}e^{-c_\al\gam2^{\ell\al}}\\
				&\ls e^{-(c_\al/2)\gam2^{\ell\al}}\no{\xi}^{-|\be_1|}\no{\eta}^{-|\be_2|}.\label{tausim1:est}
	\end{align}

Similarly, for $\s\geq1/2$, using \req{tausim1} instead, we obtain
	\begin{align}
		\notag|\bdy^\be e^{\gam R_{\al,\s}(\xi,\eta)}|&\ls\sum_{\mu,\nu}C_{\mu,\nu} \gam^{ |\mu|}e^{-(1-\s^\al)\gam\no{\xi}^\al}\prod_{1\leq |b|\leq|\be|}(\bdy^bR_{\al,\s}(\xi,\eta))^{\nu_b}\\
		\notag&\ls\sum_{\mu,\nu} C_{\mu,\nu}\gam^{|\mu|}e^{-c_\al\gam2^{\ell\al}}\prod_{1\leq |b|\leq|\be|}(2^{\ell(\al-|b_\xi|)}2^{-k|b_\eta|})^{\nu_b}\\
		&\ls e^{-(c_\al/2)\gam2^{\ell\al}}\no{\xi}^{-|\be_1|}\no{\eta}^{-|\be_2|}.\label{tausim0:est}
	\end{align}

For the other factors, observe that since $\no{\xi\s+\eta}\sim2^j$ we have
	\begin{align}
		%\left|\bdy^\be_\xi\no{\xi}^{-\de}\right|&\ls\no{\xi}^{-\de-|\be|}\ls2^{-\ell \de}\no{\xi}^{-|\be|}\\
		\left|\bdy^\be\no{\xi\s+\eta}^{\al-2}\right|&\ls\no{\xi\s+\eta}^{\al-2-|\be|}\ls2^{j(\al-2)}\no{\xi}^{-|\be_\xi|}\no{\eta}^{-|\be_\eta|}\label{comm:factor2}\\
	\left|\bdy^\be({\xi_i\s+\eta_i})\right|&\ls\begin{cases}
							2^\ell+2^k&, |\be_\xi|=0\\
							1&, |\be|=1\ \text{and}\ |\be_{\xi}^i|\ \text{or}\ |\be_\eta^i|=1\\
							0&,|\be|\geq2\ \text{or}\ |\be^{i'}|\neq0\ i'\neq i
		\end{cases}\label{comm:factor3}%\no{\xi\s+\eta}^{-|\be|}\ls2^{-j|\be|}\ls\no{\xi}^{-|\be|}\\
	\end{align}
It follows from \req{comm:factor2} and \req{comm:factor3} that
	\begin{align}\label{comm:factor:final}
		\left|\bdy^\be\left(\no{\xi\s+\eta}^{\al-2}(\xi_i\s+\eta_i)\right)\right|\ls2^{j(\al-1)}2^{-\ell|\be_\xi|}2^{-k|\be_\eta|}.
	\end{align}
We also have
	\begin{align}
	\left|\bdy_\xi^\be\ph_\ell(\xi)\right|&\ls2^{-\ell|\be|}\ls\no{\xi}^{-|\be|}\label{comm:factor4},\\
	\left|\bdy_\eta^\be\ph_k(\eta)\right|&\ls2^{-k|\be|}\ls\no{\eta}^{-|\be|}\label{comm:factor5}
	\end{align}
for all $\eta\in\R^2$.

Therefore, combining \req{tausim1:est}, \req{tausim0:est} and \req{comm:factor:final}-\req{comm:factor5}, we can deduce that
	\begin{align}
		\left|\bdy^{\be_1}_\xi\bdy^{\be_2}_\eta {m}_{A}(\xi,\eta)\right|&\ls\gam2^{-j(1-\al)}e^{-(c_\al/2)\gam2^{\ell\al}}\no{\xi}^{-|\be_1|}\no{\eta}^{-|\be_2|}\notag\\
		&\ls\gam^{1-\de/\al}2^{-j(1-\al)}2^{-\ell\de}\no{\xi}^{-|\be_1|}\no{\eta}^{-|\be_2|}.\label{mA:est},
	\end{align}
for any $\de>0$.

On the other hand, we can estimate $m_B$ using \req{tausim1:est} and \req{tausim0:est} by
	\begin{align}
		\left|\bdy_\xi^{\be_1}\bdy_\eta^{\be_2} m_B(\xi,\eta)\right|\ls 2^{-j}\no{\xi}^{-|\be_1|}\no{\eta}^{-|\be_2|}.
	\end{align}
Let $N>1$ and $p^*=(pN)/(N-1)$.  Then by Theorem \ref{axis:thm} and the Bernstein inequalities
	\begin{align}\label{comm:estA}
		\Sob{T_{m_{A}}(\lpl\bdy_i\til{f},\lpk \til{g})}{L^p}&\ls\gam^{1-\de/\al}2^{-j(1-\al)}2^{(1-\de +2/p^*)\ell}\Sob{\lpl\til{f}}{L^p}\Sob{\lpk \til{g}}{L^{p^*}},\\
		\Sob{T_{m_{B}}(\lpl\bdy_i\til{f},\lpk \til{g})}{L^p}&\ls2^{-j}2^{(1+2/p^*)\ell}\Sob{\lpl\til{f}}{L^p}\Sob{\lpk \til{g}}{L^{p^*}}.\label{comm:estB}
	\end{align}

Suppose $s<1+2/p$ and choose $N,\de>0$ so that $s<1+2/p^*-\de$.  From \req{comm:estB}, we apply the Bernstein inequalities and the fact that $|k-j|\leq3$ to get
	\begin{align}\label{ksimj:commB}
		\lVert T_{m_{B}}&(\lpl\bdy_i\til{f},\lpk \til{g})\rVert_{L^p}\\
			\notag&\ls2^{-(s+t-2/p)j}\Sob{\til{g}}{\dot{B}^t_{p,\infty}}\sum_{k\geq\ell+1}2^{-(1+2/p^*-s)(k-\ell)}2^{s\ell}\Sob{\lpl\til{f}}{L^p}\\
			\notag&\ls2^{-(s+t-2/p)j}C_j\Sob{\til{f}}{\dot{B}^s_{p,q}}\Sob{\til{g}}{\dot{B}^t_{p,\infty}},
	\end{align}
where	
	\begin{align}\notag
		C_j:=\sum_{j\geq\ell-2}2^{-(1+2/p^*-s)(j-\ell)}2^{s\ell}\Sob{\lpl\til{f}}{L^p},
	\end{align}
which satisfies $(C_j)_{j\in\Z}\in\ell^q$ since $s<1+2/p^*$.

Similarly, since $s<1+2/p^*-\de$, from \req{comm:estA} we can estimate
	\begin{align}\label{ksimj:commA}
		&\Sob{T_{m_{A}}(\lpl\bdy_i\til{f},\lpk \til{g})}{L^p}\\
			\notag&\ls\gam^{1-\de/\al}2^{-(\de-\al+s+t-2/p)j}\Sob{\til{g}}{\dot{B}^t_{p,\infty}}\sum_{k\geq\ell+1}2^{-(1+2/p^*-\de-s)(k-\ell)}2^{s\ell}\Sob{\lpl\til{f}}{L^p}\\
			\notag&\ls\gam^{(\al-\de)/\al}2^{(\al-\de)j}2^{-(s+t-2/p)j}C_j\Sob{\til{f}}{\dot{B}^s_{p,q}}\Sob{\til{g}}{\dot{B}^t_{p,\infty}},
	\end{align}
where	
	\begin{align}\notag
		C_j:=\sum_{j\geq\ell-2}2^{-(1+2/p^*-\de-s)(j-\ell)}2^{s\ell}\Sob{\lpl\til{f}}{L^p},
	\end{align}
which satisfies $(C_j)_{j\in\Z}\in\ell^q$ since $s<1+2/p^*-\de$.

The estimates \req{kgtrj:diag}, \req{kgtrj:holder}, \req{ksimj:holder}, \req{ksimj:commB}, and \req{ksimj:commA} prove Theorem \ref{thm:comm:est}.

%%%%%%%%%%%%%%%%%%%%%%%%%%%%%%%%%%%%%%%%%%%%%

\section{Proof of Theorem \ref{main:comm:thm}}\label{sect:apriori}

We consider the sequence of approximate solutions $\tht^n$ determined by
	\begin{align}\label{qg:approx}
		\begin{cases}
			\bdy_t\tht^{n+1}+\Lam^\kap\tht^{n+1}+u^n\cdotp\del\tht^{n+1}=0\ \text{in}\ \R^2\times\R_+,\\
			u^n=(-R_2\tht^{n+1},R_1,\tht^n)\ \text{in}\ \R^2\times\R_+,\\
			\tht^{n+1}\big{|}_{t=0}=\tht_0\ \text{in}\ \R^2,
		\end{cases}
	\end{align}
for $n=1,2,\dots$, and where  $\tht^0$ satisfies the heat equation
	\begin{align}
		\begin{cases}
			\bdy_t\tht^0+\Lam^\kap\tht^0=0\ \text{in}\ \R^2\times\R,\\
			\tht^{0}\big{|}_{t=0}=\tht_0\ \text{in}\ \R^2.
		\end{cases}
	\end{align}
It is well-known that $\tht^n$ is Gevrey regular for $n\geq0$.  In particular, we may define
	\begin{align}
		\til{\tht}^n(s):=G_\gam\tht^{n},\ \text{and}\ \til{u}^n(s):=G_\gam u^n(s),
	\end{align}
where $\gam=\gam(s):=\lam s^{\al/\kap}$.  By Theorem 1.2 in \cite{cmz}, the sequence $(\tht^n)_{n\geq0}$ converges in $\dot{B}^{\s}_{p,q}$, where $\s:=1+2/p-\kap$, to some function $\tht\in C([0,T];\dot{B}^\s_{p,q})$, which is a solution of \req{qg}, provided that either $T$ or $\Sob{\tht_0}{\dot{B}^\s_{p,q}}$ is sufficiently small.  Additionally, we will show that the approximating sequence satisifies
	\begin{align}\label{unif:bounds}
		\sup_{0<t<T}t^{\be/\kap}\Sob{\tht^n(t)}{\dot{B}^{\s+\be}_{p,q}}\ls \Sob{\tht_0}{\dot{B}^{\s}_{p,q}}\ \text{and}\ \lim_{T\goesto0}\sup_{n\geq0}\Sob{\tht^n(t)}{\dot{B}^{\s+\be}_{p,q}}=0,
	\end{align}
for any $0<\be<\kap/2$ and $n\geq0$, where the suppressed constant above is independent of $n$.  Hence, once \req{unif:bounds} is shown, then it will suffice to obtain \textit{a priori} bounds for $\Sob{{\tht}^n(\ \cdotp)}{X_T}$, independent of $n$.

\subsection{Part I: Approximating sequence}  To prove \req{unif:bounds}, we follow \cite{miura}.  First observe that by Lemma \ref{lin:heat:eqn} we have
	\begin{align}
		\notag\Sob{e^{-t\Lam^\kap}\tht^0}{\Bess}=\Sob{e^{-\gam\Lam^\al}e^{-t\Lam^\kap}\til{\tht}^0}{\Bess}\leq\Sob{e^{-t\Lam^\kap}\til{\tht}^0}{\Bess}\leq\Sob{\tht^0}{X_T},
	\end{align}
for all $0<t<T$.  Then Lemma \ref{lem:heat:ker} and \ref{lin:heat:eqn} imply that
	\begin{align}
		\notag t^{\be/\kap}\Sob{e^{-t\Lam^\kap}\tht^0}{\Bess}\ls \Sob{\tht_0}{\Bessig}\ \text{and}\ \lim_{T\goesto0}\sup_{0<t<T}t^{\be/\kap}\Sob{e^{-t\Lam^\kap}\tht^0}{\Bess}=0.
	\end{align}
We proceed by induction.  Assume that \req{unif:bounds} holds for some $n>0$.

We apply $\lpj$ to \req{qg:approx} to obtain
	\begin{align}\label{qg:approx1}
		\bdy_t\tht_j^{n+1}+\Lam^\kap\tht_j^{n+1}+\lpj(u^n\cdotp\del\tht^{n+1})=0.
	\end{align}
Then we take the $L^2$ inner product of \req{qg:approx1} with $|\tht_j|^{p-2}\tht_j$ and use the fact that $\del\cdotp u^n=0$ to write
	\begin{align}\label{qg:approx1}
		\frac{1}p\ddt\Sob{{\tht}_j^{n+1}}{L^p}^p&+\int_{\R^2}\Lam^\kap{\tht}_j^{n+1}|{\tht}_j^{n+1}|^{p-2}{\tht}_j^{n+1}\ dx\\
		\notag&=-\int_{\R^2}[\lpj,u^n]\del{\tht}_j^{n+1}|{\tht}_j^{n+1}|^{p-2}{\tht}_j^{n+1}\ dx.
	\end{align}
Note that we used the fact that 
	\begin{align}\label{cancellation}
		\int_{\R^2}u^n\cdotp\del{\tht}_j^{n+1}|{\tht}_j|^{p-2}\til{\tht}_j^{n+1}\ dx=0,
	\end{align}
which one obtains by integrating by parts and invoking the fact that $\del\cdotp u^n=0$ for all $n>0$.  Now, we apply Lemma \ref{pos:lem}, Lemma \ref{gen:bern}, and H\"older's inequality, so that after dividing by $\Sob{\tht_j}{L^p}^{p-1}$, \req{qg:approx1} becomes
	\begin{align}
		\notag\ddt\Sob{{\tht}_j^{n+1}}{L^p}&+C2^{\kap j}\Sob{{\tht}_j^{n+1}}{L^p}\ls\Sob{[\lpj,u^n]\del\tht^{n+1}}{L^p}.
	\end{align}
Let $\be<\kap/2$.  By Corollary \ref{comm:besov} with $s=\s+\be$ and $t=2/p-\kap+\be$ we get
	\begin{align}
		\notag\ddt\Sob{{\tht}_j^{n+1}}{L^p}&+C2^{\kap j}\Sob{{\tht}_j^{n+1}}{L^p}\ls2^{-((\s+\be)-(\kap-\be))j}c_j\Sob{\tht^{n}}{\Bess}\Sob{\tht^{n+1}}{\Bess}.
	\end{align}
Note that we have used boundedness of the Riesz transform.  Thus, multiplying by $2^{(\s+\be)j}$, then applying Gronwall's inequality gives
	\begin{align}
	\notag	&\Sob{\tht^{n+1}(t)}{\Bess}\ls \left(\sum_j\left(e^{-C2^{\kap j}t}2^{(\s+\be)j}\Sob{\lpj \tht_0}{L^p}\right)^q\right)^{1/q}\\
	\notag		&+\left(\int_0^t\sum_j\left(e^{-C2^{\kap j}(t-s)}2^{(\kap-\be)j}c_j\Sob{\tht^{n}(s)}{\Bess}\Sob{\tht^{n+1}(s)}{\Bess}\ ds\right)^q\right)^{1/q}.
	\end{align}
In particular, this implies 
	\begin{align}
	\notag	t^{\be/\kap}\Sob{\tht^{n+1}(t)}{\Bess}&\ls t^{\be/\kap} \left(\sum_j\left(e^{-C2^{\kap j}t}2^{(\s+\be)j}\Sob{\lpj \tht_0}{L^p}\right)^q\right)^{1/q}\\
	\notag		&+t^{\be/\kap}\left(\int_0^ts^{-2\be/\kap}(t-s)^{-(1-\be/\kap)}ds\right)\left(\sum_jc_j^q \right)^{1/q}\\
				&\times\left(\sup_{0<t<T}t^{\be/\kap}\Sob{\tht^{n}(s)}{\Bess}\right)\left(\sup_{0<t<T}t^{\be/\kap}\Sob{\tht^{n+1}(s)}{\Bess}\right),
	\end{align}
where we have used the fact that 
	\begin{align}\label{exp:decay}
		x^be^{-ax^c}\ls a^{-b/c}.
	\end{align} 
Since $\be<\kap/2$,  $(c_j)_{j\in\Z}\in\ell^q$ and 
	\begin{align}
	\notag	\int_0^t\frac{1}{s^{2\be/\kap}(t-s)^{1-\be/\kap}}\ ds\ls t^{-\be/\kap},
	\end{align}
we actually have
	\begin{align}\label{induction}
		\sup_{0<t<T}t^{\be/\kap}\Sob{\tht^{n+1}(t)}{\Bess}&\ls\sup_{0<t<T} t^{\be/\kap} \left(\sum_j\left(e^{-C2^{\kap j}t}2^{(\s+\be)j}\Sob{\lpj \tht_0}{L^p}\right)^q\right)^{1/q}\\
	\notag		+&\left(\sup_{0<t<T}t^{\be/\kap}\Sob{\tht^{n}(t)}{\Bess}\right)\left(\sup_{0<t<T}t^{\be/\kap}\Sob{\tht^{n+1}(t)}{\Bess}\right),
	\end{align}
In fact, \req{exp:decay} also implies
	\begin{align}\notag
		M(t):= t^{\be/\kap}\left(\sum_j\left(e^{-C2^{\kap j}t}2^{(\s+\be)j}\Sob{\lpj \tht_0}{L^p}\right)^q\right)^{1/q}\ls \Sob{\tht_0}{\Bessig}.
	\end{align}
From Lemma \ref{lin:heat:eqn} we know that
	\begin{align}\notag
		e^{-C2^{\kap j}t}\Sob{\lpj \tht_0}{L^p}\ls\Sob{e^{-c't\Lam^\kap}\lpj\tht_0}{L^p},
	\end{align}
for some ${c'}>0$, where $v_j=e^{-{c'}t\Lam^\kap}\lpj\tht_0$ solves the heat equation
	\begin{align}\notag
		\begin{cases}
			\bdy_tv+c'\Lam^\kap v=0\\
				v(x,0)=\lpj \tht_0(x).
		\end{cases}
	\end{align}
Hence
	\begin{align}\notag
		 M(t)\ls \sup_{0<t<T}t^{\be/\kap}\Sob{e^{-c't\Lam^\kap}\tht_0}{\dot{B}_{p,q}^{\s+\be}},
	\end{align}
so that by arguing as in the case $n=0$, we may deduce that
	\begin{align}\label{Mgoesto0}
		\lim_{T\goesto0}\sup_{0<t<T}M(t)=0.
	\end{align}
Recall that by hypothesis, we have
	\begin{align}\notag
		\lim_{T\goesto0}\sup_{n\geq0}\Sob{\tht^n(t)}{\dot{B}^{\s+\be}_{p,q}}=0
	\end{align}
Then returning to \req{induction}, we may choose $T$ sufficiently small so that
	\begin{align}\notag
		\sup_{0<t<T}t^{\be/\kap}\Sob{\tht^{n+1}(t)}{\Bess}\ls \sup_{0<t<T}M(t).
	\end{align}
Finally, invoking \req{Mgoesto0} completes the induction.
%We claim that
%	\begin{align}
%		\sup_{0<t<T}t^{\be/\kap}\Sob{\tht^{n+1}(t)}{\dot{B}^{\s+\be}_{p,q}}\ls M(T).
%	\end{align}
%Indeed, from \req{qg:approx2} we have
%	\begin{align}
%\notag \Sob{\tht^{n+1}(t)}{\Bess}&\ls e^{-C2^{\kap j}t}\Sob{\tht_0}{\Bessig}\\
%	\notag		&+\int_0^te^{-C2^{\kap j}(t-s)}2^{(\kap-\be)j}\Sob{\tht_j^{n}(s)}{\Bess}\Sob{\tht^{n+1}_j(s)}{\Bess}\ ds.
%	\end{align}

\subsection{A priori bounds} First apply $G_\gam\lpj$ to \req{qg:approx}.  Using the fact that $G_\gam,\lpj,\del$ are Fourier multipliers (and hence, commute), we obtain
	\begin{align}\label{qggev:approx}
		\bdy_t\til{\tht}_j^{n+1}+\Lam^\kap\til{\tht}_j^{n+1}+G_\gam\lpj(u^n\cdotp\del\tht^{n+1})=\lam^{\kap/\al}\gam^{1-\kap/\al}{\Lam^\al\til{\tht}_j^{n+1}}.
	\end{align}
%Now observe that by definition $\del\cdotp u^n=0$.  Integrating by parts then gives
%	\begin{align}\label{cancellation}
%		\int_{\R^2}u^n\cdotp\del\til{\tht}_j^{n+1}|\til{\tht}_j|^{p-2}\til{\tht}_j^{n+1}\ dx=0.
%	\end{align}
%Upon taking the $L^2$ inner product of \req{qggev:approx} with $|\til{\tht}_j^{n+1}|^{p-2}\til{\tht}_j^{n+1}$, we can use \req{cancellation} to obtain
%	\begin{align}\label{qggev:approx1}
%		\frac{1}p\ddt\Sob{\til{\tht}_j^{n+1}}{L^p}^p&+\int_{\R^2}\Lam^\kap\til{\tht}_j^{n+1}|\til{\tht}_j^{n+1}|^{p-2}\til{\tht}_j^{n+1}\ dx\\
%		\notag&=\lam^{\kap/\al}\gam^{1-\kap/\al}\int_{\R^2}{\Lam^\al\til{\tht}^{n+1}_j}|\til{\tht}_j^{n+1}|^{p-2}\til{\tht}_j^{n+1}\ dx\\
%		\notag&-\int_{\R^2}[G_\gam\lpj,u^n]\del\til{\tht}_j^{n+1}|\til{\tht}_j^{n+1}|^{p-2}\til{\tht}_j^{n+1}\ dx.
%	\end{align}
Now apply Lemma \ref{pos:lem}, Lemma \ref{gen:bern}, and H\"older's inequality, as well as Lemma \ref{lem:lin:gev} to obtain
%	\begin{align}
%		\notag\ddt\Sob{\til{\tht}_j^{n+1}}{L^p}&+2^{\kap j}\Sob{\til{\tht}_j^{n+1}}{L^p}\ls\lam^{\kap/\al}\gam^{1-\kap/\al}\Sob{\Lam^\al\til{\tht}_j^{n+1}}{L^p}+\Sob{[G_\gam\lpj,u^n]\del\tht^{n+1}}{L^p}.
%	\end{align}
\begin{align}
		\notag\ddt&\Sob{\til{\tht}^{n+1}_j}{L^p}+C2^{\kap j}\Sob{\til{\tht}_j^{n+1}}{L^p}\\
			\notag&\ls\lam^{\kap/\al}\gam^{1-\kap/\al}\Sob{\Lam^\al{\tht}_j^{n+1}}{L^p}+\lam^{\kap/\al}\Sob{\Lam^\kap\til{\tht}^{n+1}_j}{L^p}+\Sob{[G_\gam\lpj,u^n]\del\tht^{n+1}}{L^p}.
	\end{align}
We choose $\lam>0$ small enough so that Lemma \ref{bern} implies
	\begin{align}
		\notag\ddt\Sob{\til{\tht}^{n+1}_j}{L^p}&+C2^{\kap j}\Sob{\til{\tht}_j^{n+1}}{L^p}\ls\gam^{1-\kap/\al}2^{\al j}\Sob{{\tht}_j^{n+1}}{L^p}+\Sob{[G_\gam\lpj,u^n]\del\tht^{n+1}}{L^p}.
	\end{align}
Now let $\be,\gam,\de>0$ such that
	\begin{align}\notag
		\de<\kap-\be<\frac{1}2+\frac{1}p.
	\end{align}
Then by Theorem \ref{thm:comm:est} with $s=\s+\be$ and $t=2/p-\kap+\be$, we have
	\begin{align}
		\notag\ddt\Sob{\til{\tht}^{n+1}_j}{L^p}+&2^{\kap j}\Sob{\til{\tht}^{n+1}_j}{L^p}\\
			\notag\ls&\gam^{1-\kap/\al}2^{\al j}\Sob{{\tht}_j^{n+1}}{L^p}\\
			\notag&+2^{-((\s+\be)-(\kap-\be))j}C_j\gam^{(\al-\de)/\al}2^{(\al-\de)j}\Sob{\til{\tht}^n}{\dot{B}_{p,q}^{\s+\be}}\Sob{\til{\tht}^{n+1}}{\dot{B}^{\s+\be}_{p,q}}\\
			\notag&+2^{-((\s+\be)-(\kap-\be))j}C_j\Sob{\til{\tht}^n}{\dot{B}_{p,q}^{\s+\be}}\Sob{\til{\tht}^{n+1}}{\dot{B}^{\s+\be}_{p,q}},
	\end{align}
Now by Gronwall's inequality, for $t\geq0$ we have
	\begin{align}\notag
		2^{(\s+\be) j}&\Sob{\til{\tht}^{n+1}_j(t)}{L^p}\ls 2^{\be j}e^{-C2^{\kap j}t}2^{\s j}\Sob{\lpj \tht_0}{L^p}\\
			\notag&+\int_0^t\gam(s)^{1-\kap/\al}2^{\al j}e^{-C(t-s)2^{\kap j}}2^{(\s+\be) j}\Sob{\tht^{n+1}_j(s)}{L^p}\ ds\\
				\notag	&+C_j\int_0^t\gam(s)^{(\al-\de)/\al}2^{(\al-\de+\kap-\be)j}e^{-C(t-s)2^{\kap j}}\Sob{\til{\tht}^n(s)}{\dot{B}_{p,q}^{\s+\be}}\Sob{\til{\tht}^{n+1}(s)}{\dot{B}^{\s+\be}_{p,q}}\ ds\\
				\notag	&+C_j\int_0^t2^{(\kap-\be)j}e^{-C(t-s)2^{\kap j}}\Sob{\til{\tht}^n(s)}{\dot{B}_{p,q}^{\s+\be}}\Sob{\til{\tht}^{n+1}(s)}{\dot{B}^{\s+\be}_{p,q}}.
	\end{align}
Substituting $\gam(s)=\lam s^{\al/\kap}$, applying the decay properties of the heat kernel $e^{-C(t-s)2^{\kap j}}$, Minkowski's inequality, and by defintion of the space $X_T$, we arrive at
	\begin{align}\notag
		\Sob{\til{\tht}^{n+1}(t)}{\Bess}\ls& t^{-\be/\kap}\Sob{\tht_0}{\dot{B}^\s_{p,q}}\\
			\notag&+\left(\int_0^ts^{-(1-(\al-\be)/\kap)}(t-s)^{-\al/\kap}\ ds\right)\left(\sup_{0<t\leq T}t^{\be/\kap}\Sob{\tht^{n+1}(t)}{\Bess}\right)\\
				\notag	&+\left(\int_0^ts^{(\al-\de-2\be)/\kap}(t-s)^{-(\al-\de+\kap-\be)/\kap }\ ds\right)\Sob{\tht^n}{X_T}\Sob{\tht^{n+1}}{X_T}\\
				\notag	&+\left(\int_0^ts^{-2\be/\kap}(t-s)^{-(\kap-\be)/\kap}\ ds\right)\Sob{\tht^n}{X_T}\Sob{\tht^{n+1}}{X_T}
	\end{align}
Provided that $\be<\kap/2$, we deduce that
	\begin{align}\label{induct}
		\Sob{\tht^{n+1}}{X_T}\leq & C_1\Sob{\tht_0}{\dot{B}^\s_{p,q}}+C_2\Sob{\tht^n}{X_T}\Sob{\tht^{n+1}}{X_T},
	\end{align}
for some constants $C_1,C_2>1$.  By Lemma \ref{lin:heat:eqn} we have
	\begin{align}
		\Sob{\tht^0}{X_T}\leq C_3\Sob{\tht_0}{\dot{B}^\s_{p,q}},
	\end{align}
for some constant $C_3>1$.  Assume that $\Sob{\tht_0}{\dot{B}^\s_{p,q}}\leq(4C)^{-1}$, where $C:=(C_1\vee C_3)C_2$.  If $\Sob{\tht^n}{X_T}\leq(2C_2)^{-1}$, then from \req{induct}, we get
	\begin{align}
		\frac{1}2\Sob{\tht^{n+1}}{X_T}\leq(4C_2)^{-1}.
	\end{align}
Therefore, by induction $\Sob{\tht^{n+1}}{X_T}\leq(2C_2)^{-1}$ for all $n\geq0$.

\section{Appendix}\label{sect:app}

\begin{lem}\label{concavity}
Let $\al<1$ and $f:\R^2\times\R^2\goesto\R$ be given by
	\begin{align}
		f(\xi,\eta):=\no{\xi}^\al+\no{\eta}^\al-\no{\xi+\eta}^\al.
	\end{align}
If $\no{\xi}/\no{\eta}\geq c$ for some $c>0$, then there exists $\eps>0$, depending only on $c$, such that $f(\xi,\eta)\geq\eps\no{\eta}^\al$.
\end{lem}

\begin{proof}
Observe that
	\begin{align}\notag
		f(\xi,\eta)=\no{\eta}^\al\left(\left\lVert\frac{{\xi}}{\no{\eta}}\right\rVert^{\al}+1-\left\lVert\frac{\xi}{\no{\eta}}+\frac{\eta}{\no{\eta}}\right\rVert^\al\right).
	\end{align}
Also observe that if $R$ is a rotation matrix, then $f(R\xi,R\eta)=f(\xi,\eta)$.  Thus, 
%If $\zeta:=\xi/\no{\eta}$ and $e:=\eta/\no{\eta}$, then
%	\begin{align}
%		f(\xi,\eta)=\no{\eta}^\al\left(\no{\zeta}^\al+1-\no{\zeta+e}^\al\right).
%	\end{align}
we may assume that $\no{\xi}\geq c$ and that $\eta=e_1$, where $e_1:=(1,0)$.  Now observe that
	\begin{align}
		\notag f(\xi,\eta)&=(\xi_1^2+\xi_2^2)^{\al/2}+1-((\xi_1+\eta_1)^2+(\xi_2+\eta_2)^2)^{\al/2}\\
			\notag	&=(\xi_1^2+\xi_2^2)^{\al/2}+1-((\xi_1+1)^2+\xi_2^2)^{\al/2}.
	\end{align}
Let $x:=\no{\xi}$.  Then
	\begin{align}\notag
		f(\xi,\eta)=g_{\xi_1}(x):=x^{\al}+1-(x^2+1+2\xi_1)^{\al/2},
	\end{align}
where $x\geq c$.  Thus, we may assume $\xi_2=0$ and $\xi_1\geq c$.  Finally, elementary calculation shows that $g(x):=|x|^\al+1-|x+1|^\al\geq \min\{g(-c),g(c)\}>0$, provided that $|x|\geq c$.
\end{proof}

Now we prove of Theorem \ref{axis:thm}

\begin{proof}
Since $m(\xi,\ \cdotp)$ is compactly supported, we may write
	\begin{align}
		{m}(\xi,\eta)=\sum_{k\in\Z^d} \hat{m}_k(\xi)e^{ik\cdotp\eta}\chi(\eta),
	\end{align}
for some smooth function $\chi$ on $\R^d$ such that $\spt\chi\sub[1/4\ls\no{\eta}\ls4]$ and $\chi(\eta)=1$ for $\eta\in[1/2\leq\no{\eta}\leq 2]$, where $\hat{m}_k(\xi):=\hat{m}(\xi,k)$ is the $k$-th Fourier coefficient of $m$.

Observe that integrating by parts gives
	\begin{align}\notag
		\hat{m}_k(\xi)=\int e^{-ik\cdotp \eta} m(\xi,\eta)\ d\eta=c_\al(-ik)^{-|\al|}\til{m}_{k,\al}(\xi),
	\end{align}		 
for all $\al\in\N^d$, where
	\begin{align}\notag
		\til{m}_{k,\al}(\xi):=\int e^{-ik\cdotp\eta}\bdy^\al_\eta (m(\xi,\eta)\chi(\eta))\ d\eta.
	\end{align}
Now by \req{hm:cond} and the fact that $\chi$ is compactly supported, we have
	\begin{align}\label{m:til}
		\notag\left|\bdy_\xi^\be\til{m}_k(\xi)\right|&\ls\sum_{\al_1+\al_2=\al}c_\al\int\left|\bdy^\be_\xi\bdy^{\al_1}_\eta m(\xi,\eta)\bdy^{\al_2}_\eta\chi(\eta)\right|\ d\eta\\
			&\ls_{\be,\al,d}\no{\xi}^{-|\be|}\int_{[1/4\ls\no{\eta}\ls4]} \no{\eta}^{-|\al_1|}\ d\eta\notag\\
			&\ls_{\be,\al,d}\no{\xi}^{-|\be|}.
	\end{align}
Thus $\til{m}_{k,\al}$ is a H\"ormander-Mikhlin multiplier.  Note that the suppressed constant in \req{m:til} is independent of $k$.
Choose $|\al|$ large enough so that $\sum_{k\in\Z^d}(1+|k|)^{-|\al|}<\infty$.  %Note that in particular, $\til{m}_{k,\al}$ is a H\"ormander-Mikhlin multiplier.

Finally, observe that
	\begin{align}\notag
		T_m(f,g)=\sum_{k\in\Z^d}T_{m_k}(f)T_{\chi_k}(g)=c_\al\sum_{k\in\Z^d}(-ik)^{-|\al|}(T_{\til{m}_{k,\al}}f)(T_{\chi}{\tau}_{-k}g).
	\end{align}
Therefore, by Minkowski's inequality, H\"older's inequality, and the H\"ormander-Mikhlin multiplier theorem we have
	\begin{align}\notag
		\Sob{T_m(f,g)}{L^r}\ls\left(\sum_{k\in\Z^d}(1+|k|)^{-|\al|}\right)\Sob{f}{L^p}\Sob{\chi*\tau_{-k}g}{L^q}.
	\end{align}
Finally, Young's convolution inequality and translation invariance of $dx$ completes the proof.

\end{proof}
%One can also prove the following multiplier theorem, where the multiplier satisfies a slightly weaker condition than that of classical H\"ormander-Mikhlin multiplier theorem.  Although the multipliers we encounter do not satisfy these estimates, we provide a proof of it here as it may find use in other applications.

The next proposition shows that Marcinkiewicz multipliers are dilation invariant.

\begin{prop}
Let $1/r=1/p+1/q$ and $T_m:L^p\times L^q\goesto L^r$ be a bounded bilinear multiplier operator whose multplier, $m$, satisfies $m\in L^\infty(\R^d\times\R^d)$
	\begin{align}\label{mc:cond}
		\left|\bdy_\xi^{\be_1}\bdy_\eta^{\be_2}m(\xi,\eta)\right|\ls_{\be,d}\no{\xi}^{-|\be_1|}\no{\eta}^{-|\be_2|},
	\end{align}
for all $\xi,\eta\in\R^d\smod\{0\}$ and multi-indices $\be_1,\be_2\in\N^d$.  Then $T_{m_\lam}$ is also bounded with the same operator norm, where $m_\lam$ is given by
	\begin{align}\notag
		m_\lam(\xi,\eta):=m(\lam\xi,\lam\eta).
	\end{align}
\begin{proof}[Proof of claim]
 We first show that $m_\lam$ also satisifes \req{mc:cond}.  Observe that
	\begin{align}\notag
		\bdy_\xi^{\be_1}\bdy_\eta^{\be_2}m_\lam(\xi,\eta)=\lam^{|\be_1|+|\be_2|}(\bdy_\xi^{\be_1}\bdy_\eta^{\be_2}m)(\lam \xi,\lam\eta)=\no{\xi}^{-|\be_1|}\no{\eta}^{-|\be_2|}.
	\end{align}
Then since $m$ satisifes \req{mc:cond} we have
	\begin{align}\notag
		\left|\bdy_\xi^{\be_1}\bdy_\eta^{\be_2}m_\lam(\xi,\eta)\right|\ls \lam^{|\be_1|+|\be_2|}\no{\lam\xi}^{-|\be_1|}\no{\lam\eta}^{-|\be_2|}.
	\end{align}
Now we prove the  claim.  Indeed, let $f\in L^p, g\in L^q$, and $\lam>0$.  Then
	\begin{align*}
		T_{m_\lam}(f,g)(x)&=\int_{\R^d}\int_{\R^d}e^{ix\cdotp(\xi+\eta)}m_\lam(\xi,\eta)\hat{f}(\xi)\hat{g}(\eta)\ d\xi\ d\eta\notag\\
				\notag&=\int_{\R^d}\int_{\R^d}e^{ix\cdotp(\xi+\eta)}m(\lam\xi,\lam\eta)\hat{f}(\xi)\hat{g}(\eta)\ d\xi\ d\eta\notag\\
\notag				&=\int_{\R^d}\int_{\R^d}e^{i(x/\lam)\cdotp(\xi'+\eta')}m({\xi'},{\eta'})\lam^{-d}\hat{f}({\xi'}/\lam)\lam^{-d}\hat{g}({\eta'}/\lam)\ d{\xi'}\ d{\eta'}\\
	\notag			&=\int_{\R^d}\int_{\R^d}e^{i(x/\lam)\cdotp(\xi+\eta)}m(\xi,\eta)\widehat{f_\lam}(\xi)\widehat{g_\lam}(\eta)\ d\xi\ d\eta\\
				&=T_m(f_\lam,g_\lam)(x/\lam)=(T_m(f_\lam,g_\lam))_{1/\lam}(x).
	\end{align*}
This implies
	\begin{align*}
		\no{T_{m_\lam}(f,g)}_{L^r}&=\lam^{d/r}\no{T_m(f_\lam,g_\lam)}_{L^r}\\
					&\ls\lam^{d/r}\no{f_\lam}_{L^p}\no{g_\lam}_{L^q}=\lam^{d/r}\lam^{-d/p}\lam^{-d/q}\no{f}_{L^p}\no{g}_{L^q}.
	\end{align*}	
In particular, $\no{T_{m_\lam}}\leq\no{T_m}$.  On the other hand, one can similarly argue
	\begin{align*}
		\no{T_m(f,g)}_{L^r}&=\lam^{-d/r}\no{T_{m_{1/\lam}}(f_{1/\lam},g_{1/\lam})}_{L^r}\\
					&\ls\lam^{-d/r}\no{f_{1\lam}}_{L^p}\no{g_{1/\lam}}_{L^q}=\lam^{-d/r}\lam^{d/p}\lam^{d/q}\no{f}_{L^p}\no{g}_{L^q}.
	\end{align*}
Therefore $\no{T_m}\leq\no{T_{m_{1/\lam}}}$.  This completes the proof.

\end{proof}

\end{prop}

%SECTION: BIBLIOGRAPHY

\end{document}